\documentclass[11pt,a4paper]{article}
\usepackage{amsmath,amssymb,amsthm,amscd,mathrsfs}
\usepackage{indentfirst}
\usepackage[top=25mm, bottom=30mm, left=30mm, right=30mm]{geometry}
\newtheorem{Def}{\bf Definition}[subsection]
\newtheorem{Thm}[Def]{\bf Theorem}
\newtheorem{Lem}[Def]{\bf Lemma}
\newtheorem{Cor}[Def]{\bf Corollary}
\newtheorem{Pro}[Def]{\bf Proposition}
\newtheorem{Rem}[Def]{\bf Remark}

\newtheorem{ThmA}{\bf Theorem}
\title{\bf Some prime factorization results for free quantum group factors}
\author{Yusuke Isono}
\date{}
\begin{document}
\maketitle

\begin{abstract}
We prove some unique factorization results for tensor products of free quantum group factors. 
They are type III analogues of factorization results for direct products of bi-exact groups established by Ozawa and Popa. 
In the proof, we first take continuous cores of the tensor products, which satisfy a condition similar to condition (AO), and discuss some factorization properties for the continuous cores. 
We then deduce factorization properties for the original type III factors. 
We also prove some unique factorization results for crossed product von Neumann algebras by direct products of bi-exact groups. 
\end{abstract}

%%%%%%%%%%%%%%%%%%%%%%%%%%%%%%%%%%%%
\section{\bf Introduction}
%%%%%%%%%%%%%%%%%%%%%%%%%%%%%%%%%%%%

We say a $\rm II_1$ factor is \textit{prime} if it is not isomorphic to tensor products of $\rm II_1$ factors. 
The first example of such a factor was given by Popa \cite{Po83}. He proved that any free group factor $L\mathbb{F}_\infty$ (with uncountably many generators) is prime. 
In \cite{Ge96}, Ge proved that $L\mathbb{F}_n$ (with $n\geq 2$) are prime by computing Voiculescu's free entropy. 
Ozawa then proved that all free group factors are solid \cite{Oz03}, meaning that the relative commutant of any diffuse von Neumann subalgebra is amenable (namely, injective). 
Solidity immediately yields primeness of any diffuse non-amenable subalgebras. 
Ozawa's proof relied on the notion of condition (AO) (see Subsection \ref{discrete decomposition}) and $C^*$-algebraic methods. 
Peterson gave a new proof of solidity of free group factors \cite{Pe06}. 

In \cite{Ge96}, Ge asked the following question: 
\begin{itemize}
	\item Is $L\mathbb{F}_2\mathbin{\bar{\otimes}} L\mathbb{F}_2$ isomorphic to $L\mathbb{F}_2\mathbin{\bar{\otimes}} L\mathbb{F}_2\mathbin{\bar{\otimes}} L\mathbb{F}_2$? 
\end{itemize}
Here the symbol $\mathbin{\bar{\otimes}}$ means the tensor product as von Neumann algebras. 
This is an extended primness problem for free group factors, which mentions numbers of tensor components. 
The question was solved by Ozawa and Popa \cite{OP03}. 
They used a combination of a tensor product analogue of condition (AO) and Popa's intertwining techniques, and obtained a relative version of Ozawa's solidity theorem. 
As a result, they deduced the following theorem, which gave a complete answer to the problem. 
See \cite[\textrm{Section 15}]{BO08} for bi-exactness below. 

\vspace{1em}
\noindent
{\bf Factorization theorem of Ozawa and Popa.}
\textit{Let $\Gamma_i$ $(i=1,\ldots,m)$ be non-amenable, ICC, bi-exact discrete groups and $N_j$ $(j=1,\ldots,n)$ be $\rm II_1$ factors. 
If $L\Gamma_1\mathbin{\bar{\otimes}} \cdots \mathbin{\bar{\otimes}} L\Gamma_m=N_1\mathbin{\bar{\otimes}} \cdots\mathbin{\bar{\otimes}} N_n$ $(=:M)$ and $m\leq n$, then $m=n$ and there are $u\in \mathcal{U}(M)$, $\sigma\in  \mathfrak{S}_n$, and $t_i>0$ with $t_1\cdots t_n=1$ such that
$uN_{\sigma(i)}u^*=L\Gamma_i^{t_i}$ for a fixed decomposition $M=L\Gamma_1^{t_1}\mathbin{\bar{\otimes}}\cdots\mathbin{\bar{\otimes}} L\Gamma_n^{t_n}$.
}
\vspace{1em}

Here recall that for a $\rm II_1$ factor $M$ and $t>0$, the \textit{amplification} $M^t$ is defined (up to $*$-isomorphism) as $pMp\mathbin{\bar{\otimes}} \mathbb{M}_n$ for any $n\in\mathbb{N}$ with $t\leq n$ and any projection $p\in M$ with trace $t/n$. 
We also recall that $\rm II_1$ factors $M$ and $N$ are \textit{stably isomorphic} if $M^t\simeq N^s$ for some $t,s>0$. 
For any $\rm II_1$ factors $M_i$ and $t>0$, $M_1\mathbin{\bar{\otimes}} M_2 \simeq (pM_1p\mathbin{\bar{\otimes}} qM_2q)\mathbin{\bar{\otimes}} \mathbb{M}_n\mathbin{\bar{\otimes}} \mathbb{M}_m\simeq M_1^t\mathbin{\bar{\otimes}} M_2^{1/t}$ for some large $n,m$ and projections $p,q$ with traces $t/n$ and $1/mt$. 
So any $\rm II_1$ factor tensor decomposition is determined up to amplifications of tensor components. 
The theorem above then means the uniqueness of the tensor decomposition up to stable isomorphism. 

In the present paper, we study similar factorization results for free quantum group factors. 
It is known that these factors satisfy condition (AO) \cite{Ve05}\cite{VV05}\cite{VV08}, and in fact tensor products of these factors satisfy an analogue of condition (AO) mentioned above for free group factors (see Proposition \ref{AO for tensor algebras}). 
So Ozawa--Popa's factorization result is true if each tensor component is a non-amenable $\rm II_1$ factor. 
However these factors often become type III and, in the general case, Popa's techniques are no loner available. 

To avoid the difficulty, we take continuous cores. 
A condition (AO) type phenomenon on cores of these factors was already observed in \cite{Is12_2}, and we generalize it to cores of the tensor products. 
This enables us to discuss some factorization properties on the continuous cores. 
In particular, we deduce some one-to-one correspondence with respect to Popa's embedding $\preceq$ (see Subsection \ref{Popa's intertwining techniques}) between tensor components on the cores. 
We then turn to see original type III algebras and deduce some factorization results. 
Thus we obtain the following theorem which is the main conclusion of the paper. 
See Definition \ref{class C} for the class $\mathcal{C}$, which contains (duals of) free quantum groups, and Subsection \ref{discrete decomposition} for type $\rm III_1$ factors, Sd-invariants, continuous cores $C_{\phi}(N)$, and centralizer algebras $(N_i)_{\phi_i}$ below. 

\begin{ThmA}\label{A}
Let $\hat{\mathbb{G}}_i$ $(i=1,\ldots,m)$ be discrete quantum groups in $\mathcal{C}$ and $N_j$ $(j=1,\ldots, n)$ be non-amenable von Neumann algebras which admit almost periodic states. 
Assume that there is an inclusion $N:=N_1\mathbin{\bar{\otimes}} \cdots \mathbin{\bar{\otimes}} N_n\subset L^\infty(\mathbb{G}_1)\mathbin{\bar{\otimes}}\cdots\mathbin{\bar{\otimes}} L^\infty(\mathbb{G}_m)=:M$ with a faithful normal conditional expectation. Then we have $n\leq m$. 

Assume further $n=m$, $N=M$, and the following conditions:
\begin{itemize}
	\item Each $L^\infty(\mathbb{G}_i)$ is a factor of type $\rm II_1$ or $\rm III_1$ and its Haar state $h_i$ is $\mathrm{Sd}(L^\infty(\mathbb{G}_i))$-almost periodic. Write $h:=h_1\otimes\cdots\otimes h_n$.
	\item Each $N_i$ is a factor of type $\rm II_1$ or $\rm III_1$ and any $\rm III_1$ factor $N_i$ admits an almost periodic state $\phi_i$ such that $(N_i)_{\phi_i}'\cap N_i=\mathbb{C}$ (put the trace as $\phi_i$ when $N_i$ is a $\rm II_1$ factor). Write $\phi:=\phi_1\otimes\cdots\otimes \phi_n$.
\end{itemize}
Then under the canonical isomorphism $C_\phi(N)\simeq C_h(M)$ with the canonical trace $\mathrm{Tr}$, there exists a unique $\sigma\in  \mathfrak{S}_n$ such that
\begin{equation*}
pC_{\phi_i}(N_i)p \preceq_{C_h(M)} C_{\sigma(i)}(L^\infty(\mathbb{G}_{\sigma(i)}) \quad (i=1,\ldots,n)
\end{equation*}
for any projection $p\in L\mathbb{R}\subset C_\phi(N)$ with $\mathrm{Tr}(p)<\infty$. In this case, $N_i$ and $L^\infty(\mathbb{G}_{\sigma(i)})$ are isomorphic when $N_i$ is a $\rm III_1$ factor, and stably isomorphic when $N_i$ is a $\rm II_1$ factor. 
\end{ThmA}

We mention that for any $\hat{\mathbb{G}}\in \mathcal{C}$, if $L^\infty(\mathbb{G})$ is a type $\rm III_1$ factor and its Haar state $h$ is $\mathrm{Sd}(L^\infty(\mathbb{G}))$-almost periodic, 
then it satisfies $L^\infty(\mathbb{G})_h'\cap L^\infty(\mathbb{G})=\mathbb{C}$ (see Subsections \ref{discrete decomposition} and \ref{CQG}). 
So as a particular case, we can put $N_i=L^\infty(\mathbb{H}_i)$ for $\hat{\mathbb{H}}_i\in\mathcal{C}$ with Haar state $\phi_i$ which is a trace or $\mathrm{Sd}(L^\infty(\mathbb{H}_i))$-almost periodic.

In the paper, we also prove some unique factorization results for crossed product von Neumann algebras by direct product groups. 
In this situation, we assume that the given isomorphism preserves subalgebras on which groups act, so that we can compare direct product groups. 
We obtain the following theorem. 
In the theorem, the symbol $\rtimes$ means the crossed product von Neumann algebras. 

\begin{ThmA}\label{B}
Let $\Gamma_i$ $(i=1,\ldots,m)$ and $\Lambda_j$ $(j=1,\ldots,n)$ be non-amenable  discrete groups. Let $(A,\mathrm{Tr}_A)$ and $(B,\tau_B)$ be semifinite tracial von Neumann algebras with $\mathrm{Tr}_A|_{\mathcal{Z}(A)}$ semifinite and $\tau_B(1)=1$. Write $\Gamma:=\Gamma_1\times \cdots \times \Gamma_m$ and $\Lambda:=\Lambda_1\times \cdots \times \Lambda_n$. 
Let $\alpha$ and $\beta$ be trace preserving actions of $\Gamma$ and $\Lambda$ on $(A,\mathrm{Tr}_A)$ and $(B,\tau_B)$ respectively. 
Assume the following conditions:
\begin{itemize}
	\item There is an inclusion $B\rtimes \Lambda\subset p(A\rtimes \Gamma)p$ for a $\mathrm{Tr}_A$-finite projection $p\in \mathcal{Z}(A)$, which sends $B$ onto $pAp$. 
	\item Either that $\mathcal{Z}(A)$ is diffuse or $A$ is a $\rm II_1$ factor.
	\item Actions $\alpha$ and $\beta$ are free on $A$ and $B$ respectively. Actions $\alpha|_{\Gamma_i}$ and $\beta|_{\Lambda_j}$ are ergodic on $\mathcal{Z}(A)$ and $\mathcal{Z}(B)$ respectively for all $i$ and $j$. 
	\item All $\Gamma_i$ are bi-exact and $A$ is amenable. 
\end{itemize}
Then we have $n\leq m$. 
If moreover $n=m$, then there exists a unique $\sigma\in  \mathfrak{S}_n$ such that 
\begin{equation*}
B\rtimes \Lambda_i \preceq_{A\rtimes\Gamma} A\rtimes\Gamma_{\sigma(i)} \quad (i=1,\ldots,n).
\end{equation*}
The same conclusions are true without amenability of $A$, if $\Gamma$ is weakly amenable, $\mathrm{Tr}_A$ is finite and $p=1_A$. 
\end{ThmA}

We mention that for $B\rtimes \Lambda_i \preceq_{A\rtimes\Gamma} A\rtimes\Gamma_{\sigma(i)}$, we can find an embedding and an intertwiner of a special form, which is discussed in Subsection \ref{Intertwiners inside crossed product von Neumann algebras}.

In the final section, we give a different approach to factorization properties for bi-exact and weakly amenable group factors.

\bigskip

Throughout the paper, we always assume that discrete groups are countable, quantum group $C^*$-algebras are separable, von Neumann algebras have separable predual, and Hilbert spaces are separable.

\bigskip

\noindent 
{\bf Acknowledgement.} The author would like to thank Yuki Arano, Cyril Houdayer, Yasuyuki Kawahigashi, Narutaka Ozawa, Hiroki Sako, and Yoshimichi Ueda for fruitful conversations. 
He was supported by JSPS, Research Fellow of the Japan Society for the Promotion of Science.

%%%%%%%%%%%%%%%%%%%%%%%%%%%%%%%%%%%%
\section{\bf Preliminaries}
%%%%%%%%%%%%%%%%%%%%%%%%%%%%%%%%%%%%

%%%%%%%%%%%%%%%%%%%%%%%%%%%%%%%%%%%%
\subsection{\bf Fullness and Discrete decompositions}\label{discrete decomposition}
%%%%%%%%%%%%%%%%%%%%%%%%%%%%%%%%%%%%

In the subsection, we recall Connes's discrete decomposition and related notions. We refer the reader to \cite{Co74} (see also \cite{Dy94}).

Let $\omega$ be a free ultra filter on $\mathbb{N}$. Consider two $C^*$-algebras
\begin{eqnarray*}
A_\omega&:=&\{(x_n)_n\in \ell^\infty(M) \mid \|\phi(\cdot\hspace{0.1em} x_n)-\phi(x_n\hspace{0.1em} \cdot)\|_{M_*}\rightarrow 0 \textrm{ as  $n\rightarrow\omega$, for all }\phi\in M_* \},\\
J_\omega&:=&\{(x_n)_n\in \ell^\infty(M) \mid x_n\rightarrow 0 \textrm{ as  $n\rightarrow\omega$ in the $*$-strong topology}\}.
\end{eqnarray*}
The quotient $C^*$-algebra $A_\omega/J_\omega$ becomes a von Neumann algebra and we denote it by $M_\omega$. 
We say a factor $M$ is \textit{full} if $\mathrm{Inn}(M)$ is closed in $\mathrm{Aut}(M)$ in the u-topology, namely, the topology of pointwise norm convergence in $M_*$. 
A factor $M$ is full if and only if $M_\omega\simeq\mathbb{C}$ for some (any) ultra filter $\omega$.

Let $M$ be a von Neumann algebra and $\phi$ a faithful normal semifinite weight on $M$. 
Then the \textit{modular operator} $\Delta_\phi$ and the \textit{modular conjugation} $J_\phi$ are defined on $L^2(M,\phi)$ as a closed operator and an anti linear map. The map $J_\phi$ satisfies $M'=J_\phi MJ_\phi \simeq M^{\rm op}$ and it gives a \textit{canonical right action} of $M^{\rm op}$ on $L^2(M,\phi)$. The family $\sigma_t^\phi:=\mathrm{Ad}\Delta_\phi^{it}$ $(t\in\mathbb{R})$ gives an $\mathbb{R}$-action on $M$ called the \textit{modular action} of $\phi$. 
The \textit{continuous core} is defined as $C_\phi(M):=M\rtimes_{\sigma^\phi} \mathbb{R}$, which does not depends on the choice of $\phi$. 
We say a type III factor $M$ is of \textit{type $\rm III_1$} if $C_\phi(M)$ is a $\rm II_\infty$ factor. 
The \textit{centralizer algebra} is defined as 
\begin{equation*}
M_\phi:=\{x\in M\mid \Delta_\phi^{it}x=x\Delta_\phi^{it} \textrm{ for all }t\in\mathbb{R}\},
\end{equation*}
We say the weight $\phi$ is \textit{almost periodic} if the modular operator $\Delta_{\phi}$ is diagonalizable, namely, 
it is of the form $\Delta_\phi=\sum_{\lambda\in \mathrm{ptSp}(\Delta_\phi)}\lambda E_\lambda$, where $\mathrm{ptSp}(\Delta_\phi)\subset \mathbb{R}^*_+$ is the point spectrum of $\Delta_\phi$ and $E_\lambda$ are spectrum projections. 
For a subgroup $\Lambda\subset \mathbb{R}^*_+$, we say $\phi$ is \textit{$\Lambda$-almost periodic} if it is almost periodic and $\mathrm{ptSp}(\Delta_\phi)\subset\Lambda$. 
Any almost periodic weight $\phi$ is semifinite on $M_\phi$ and hence there is a faithful normal conditional expectation from $M$ onto $M_\phi$ \cite[\textrm{Theorem IX.4.2}]{Ta2}. 
When $M$ is a factor with an almost periodic weight, its \textit{Sd-invariant} is defined as 
\begin{equation*}
\mathrm{Sd}(M):=\bigcap_{\phi\textrm{ is almost periodic on $M$}}\mathrm{ptSp}(\Delta_\phi).
\end{equation*}
It becomes a subgroup of $\mathbb{R}^*_+$. 
When $M$ is a full type III factor, an almost periodic wight $\phi$ is $\mathrm{Sd}(M)$-almost periodic if and only if $(M_\phi)'\cap M=\mathbb{C}$.

Recall from \cite{Oz03} that a von Neumann algebra $M\subset \mathbb{B}(H)$ satisfies condition (AO) if there are $\sigma$-weakly dense $C^*$-subalgebras $A\subset M$ and $B\subset M'$ such that $A$ is locally reflexive and the map $A\otimes_{\rm alg}B\ni a\otimes b\mapsto ab\in \mathbb{B}(H)/\mathbb{K}(H)$ is bounded with respect to the minimal tensor norm. 
The following lemma is well known.
\begin{Lem}
Let $M$ be a von Neumann algebra. If $M$ is non-amenable and satisfies condition $\rm (AO)$ in $\mathbb{B}(L^2(M))$, then $C^*\{M,M'\}\cap \mathbb{K}(L^2(M))\neq0$.
\end{Lem}
\begin{proof}
Write $\mathbb{K}:=\mathbb{K}(L^2(M))$. 
Let $A\subset M$ and $B\subset M'$ be  $\sigma$-weakly dense unital $C^*$-subalgebras satisfying that the map $\nu\colon A\otimes_{\rm alg} B \ni a\otimes b \mapsto ab\in \mathbb{B}(L^2(M))/\mathbb{K}$ is bounded on $A\otimes_{\rm min}B$. 
If $C^*\{M,M'\}\cap \mathbb{K}=0$, then the image of $\nu$ is contained in $(C^*\{M,M'\}+ \mathbb{K})/\mathbb{K}\simeq C^*\{M,M'\}/(C^*\{M,M'\}\cap \mathbb{K})\simeq C^*\{M,M'\}$. 
Hence $\nu$ is bounded without the quotient of $\mathbb{K}$. 
Since $A$ is locally reflexive, $A$ is unital, and $\nu$ is normal on $A\otimes_{\rm min}\mathbb{C}$, we can extend $\nu$ on $M\otimes_{\rm min}B$. We again extend $\nu$ on $\mathbb{B}(L^2(M))\otimes_{\rm min}B$ by Arveson's theorem and denote by $\Phi$. 
The restriction of $\Phi$ on $\mathbb{B}(L^2(M))\otimes_{\rm min}\mathbb{C}$ is a conditional expectation onto $B'=M$ (since $\mathbb{C}\otimes_{\rm min}B$ is contained in the multiplicative domain of $\Phi$). Thus $M$ is amenable.
\end{proof}

The following lemma is a general version of \cite[\textrm{Corollary 2.3}]{Co76}. We thank Cyril Houdayer for demonstrating the proof of the lemma. 

\begin{Lem}\label{fullness of tensor factors}
Let $M$ and $N$ be factors. 
If $C^*\{M,M'\}\cap \mathbb{K}(L^2(M))\neq 0$, then the map $N_\omega\ni (x_n)_n\mapsto (1\otimes x_n)_n\in (M\mathbin{\bar{\otimes}} N)_\omega$ is surjective. 
In particular, $M\mathbin{\bar{\otimes}} N$ is full if $N$ is full (possibly $N=\mathbb{C}$).
\end{Lem}
\begin{proof}
Since $M$ is a factor, we have $C^*\{M,M'\}''=(M\cap M')'=\mathbb{B}(L^2(M))$. 
Let $x\in C^*\{M,M'\}\cap \mathbb{K}(L^2(M))$ and $y\in \mathbb{B}(L^2(M))$  be non-zero elements. Let $y_i\in C^*\{M,M'\}$ be a bounded net converging to $y$ strongly. 
Then the net $y_ix$ converges to $yx$ in the norm topology and hence we have $yx\in C^*\{M,M'\}$. This implies $\mathbb{B}(L^2(M))x\mathbb{B}(L^2(M))\subset C^*\{M,M'\}$ and hence $\mathbb{K}(L^2(M))\subset C^*\{M,M'\}$. 

Let $\phi$ and $\psi$ be faithful normal states on $M$ and $N$ respectively and write $H:=L^2(M,\phi)\otimes L^2(N,\psi)$, where the symbol $\otimes$ means the tensor product of Hilbert spaces. 
Let $(x_n)_n\in (M\mathbin{\bar{\otimes}} N)_\omega$, namely, $(x_n)_n$ be a bounded sequence satisfying $\lim_{n\rightarrow\omega}\|[x_n,\chi]\|_{(M\mathbin{\bar{\otimes}} N)_*}=0$ for any $\chi\in(M\mathbin{\bar{\otimes}} N)_*$.  
We will show $x_n-(\phi\otimes\mathrm{id}_N)(x_n)\rightarrow 0$ as $n\rightarrow\omega$ in the $*$-strong topology, which means $(x_n)_n=((\phi\otimes\mathrm{id}_N)(x_n))_n\in N_\omega$. 
Since $[x_n,ba]=b[x_n,a]\rightarrow 0$ $*$-strongly for $a\in M\otimes_{\rm min} \mathbb{C}$ and $b\in M'\otimes_{\rm min} \mathbb{C}$, we have $[x_n,a]\rightarrow0$ for any $a\in C^*\{M,M'\}\otimes_{\rm min}\mathbb{C}$. 
Let $P_{a,1}$ be the partial isometry from $\mathbb{C}\hat{1}$ to $\mathbb{C}\hat{a}$ for $a\in M$, which is contained in $\mathbb{K}(L^2(M))\subset C^*\{M,M'\}$. 
Then for any $a\in M$ and $b\in N$, we have 
\begin{eqnarray*}
(x_n-(\phi\otimes\mathrm{id}_N)(x_n))(\hat{a}\otimes \hat{b})
&=&x_n(P_{a,1}\otimes 1)(\hat{1}\otimes \hat{b})- (P_{a,1}\otimes 1)(\phi\otimes\mathrm{id}_N)(x_n))(\hat{1}\otimes \hat{b})\\
&=&x_n(P_{a,1}\otimes 1)(\hat{1}\otimes \hat{b})- (P_{a,1}\otimes 1)x_n(\hat{1}\otimes \hat{b})\\
&=&[x_n,(P_{a,1}\otimes 1)](\hat{1}\otimes \hat{b})\rightarrow 0
\end{eqnarray*}
in the norm topology of $H$. 
Since the same is true for $x_n^*$, we have $x_n-(\phi\otimes\mathrm{id}_N)(x_n)\rightarrow 0$ in the $*$-strong topology.
\end{proof}
\begin{Lem}\label{full tensor}
Let $M_i$ be full factors with $C^*\{M_i,M_i'\}\cap\mathbb{K}(L^2(M_i))\neq0$. Let $\phi_i$ be faithful normal semifinite weight on $M_i$ which is $\mathrm{Sd}(M_i)$-almost periodic. 
Then $\phi_1\otimes\cdots \otimes\phi_n$ is $\mathrm{Sd}(M_1\mathbin{\bar{\otimes}}\cdots\mathbin{\bar{\otimes}} M_n)$-almost periodic. In particular $\mathrm{Sd}(M_1\mathbin{\bar{\otimes}}\cdots\mathbin{\bar{\otimes}} M_n)=\mathrm{Sd}(M_1)\cdots\mathrm{Sd}(M_n)$ and  $(M_1\mathbin{\bar{\otimes}}\cdots\mathbin{\bar{\otimes}} M_n)_{\phi_1\otimes\cdots\otimes \phi_n}$ is a factor.
\end{Lem}
\begin{proof}
By the previous lemma, $M_1\mathbin{\bar{\otimes}}\cdots\mathbin{\bar{\otimes}} M_n$ is a full factor. We have 
\begin{eqnarray*}
&&(M_1\mathbin{\bar{\otimes}}\cdots\mathbin{\bar{\otimes}} M_n)_{\phi_1\otimes\cdots\otimes \phi_n}'\cap (M_1\mathbin{\bar{\otimes}}\cdots\mathbin{\bar{\otimes}} M_n) \\
&\subset& ((M_1)_{\phi_1}\mathbin{\bar{\otimes}}\cdots\mathbin{\bar{\otimes}} (M_n)_{\phi_n})'\cap (M_1\mathbin{\bar{\otimes}}\cdots\mathbin{\bar{\otimes}} M_n)\\
&=&(M_1\cap (M_1)_{\phi_1}')\mathbin{\bar{\otimes}}\cdots\mathbin{\bar{\otimes}} (M_n\cap (M_n)_{\phi_n})'=\mathbb{C}.
\end{eqnarray*}
\end{proof}

Next we recall discrete decompositions. Let $M$ be a von Neumann algebra and $\phi$ a $\Lambda$-almost periodic wight on $M$ for a countable subgroup $\Lambda\subset\mathbb{R}^*_+$. 
As a discrete group, take the Pontryagin dual of $\Lambda$ and denote by $K$. By definition, we have a map $\mathbb{R}\simeq\widehat{\mathbb{R}^*_+}\rightarrow K$, which is injective and has a dense image if $\Lambda\subset\mathbb{R}^*_+$ is dense, and which is surjective if $\Lambda\subset\mathbb{R}^*_+$ is periodic. 
The modular action of $\mathbb{R}$ extends to an action of $K$. Take the crossed product von Neumann algebra $M\rtimes K$ by the action. Then by Takesaki duality, we have $(M\rtimes K)\rtimes \Lambda\simeq M\mathbin{\bar{\otimes}}\mathbb{B}(\ell^2(\Lambda))$. 
The dual weight of $M\mathbin{\bar{\otimes}}\mathbb{B}(\ell^2(\Lambda))$ is of the form $\phi\otimes \omega$, where $\omega:=\mathrm{Tr}(\lambda \hspace{0.1em}\cdot)$ for a closed operator $\lambda$ on $\ell^2(\Lambda)$ given by $\lambda(a)=a$. 
Since $M\rtimes K$ is the fixed point algebra of the dual action on $(M\rtimes K)\rtimes \Lambda$, we have $M\rtimes K \simeq (M\mathbin{\bar{\otimes}}\mathbb{B}(\ell^2(\Lambda)))_{\phi\otimes \omega}=:D_\phi(M)$. 
When $M$ is of type III, we have $D_\phi(M)\rtimes\Lambda\simeq M$, which is called a \textit{discrete decomposition}. The subalgebra $D_\phi(M)$ is called a \textit{discrete core} of M. 
Since $\phi\otimes\omega$ is almost periodic (or $\Lambda$ is discrete), there exists a faithful normal conditional expectation from $M$ onto $D_\phi(M)$. 
When $M$ is a full factor, there always exists an $\mathrm{Sd}(M)$-almost periodic weight $\phi$ on $M$. 
In this case, since $\phi\otimes \omega$ is $\mathrm{Sd}(M)$-almost periodic on $M\mathbin{\bar{\otimes}}\mathbb{B}(\ell^2(\Lambda))\simeq M$, the discrete core $D_\phi(M)=(M\mathbin{\bar{\otimes}}\mathbb{B}(\ell^2(\Lambda)))_{\phi\otimes \omega}$ is a $\rm II_\infty$ factor, and hence is identified as $M_\phi\mathbin{\bar{\otimes}} \mathbb{B}(H)$ for some separable infinite Hilbert space $H$.

Finally $M$ is amenable if and only if $D_\phi(M)$ is amenable for any von Neumann algebra $M$ and its almost periodic weight $\phi$. 
Hence when $M$ is of type III and non-amenable, since $D_\phi(M)\simeq (M\mathbin{\bar{\otimes}}\mathbb{B}(\ell^2(\Lambda)))_{\phi\otimes \omega}$, there is an almost periodic weight on $M(\simeq M\mathbin{\bar{\otimes}}\mathbb{B}(\ell^2(\Lambda)))$ such that its centralizer algebra is non-amenable. 
By using a type decomposition, We have the same result for any von Neumann algebra with almost periodic weights. We write this observation as follows.
\begin{Lem}
Let $M$ be a non-amenable von Neumann algebra which admits an almost periodic wight. Then there exists an almost periodic weight $\phi$ such that $M_\phi$ is non-amenable.
\end{Lem}

%%%%%%%%%%%%%%%%%%%%%%%%%%%%%%%%%%%%
\subsection{\bf Compact and discrete quantum groups}\label{CQG}
%%%%%%%%%%%%%%%%%%%%%%%%%%%%%%%%%%%%

Let $\mathbb{G}$ be a compact quantum group. In the paper, we use the following notation, which are same as in our previous works \cite{Is12_2}\cite{Is13}. 
We denote the Haar state by $h$, the set of all equivalence classes of all irreducible unitary corepresentations by $\mathrm{Irred}(\mathbb{G})$, and right and left regular representations by $\rho$ and $\lambda$ respectively. 
For $x\in \mathrm{Irred}(\mathbb{G})$, $(u_{i,j}^x)_{i,j=1}^{n_x}$ are coefficients of $x$ and we frequently omit $n_x$. 
We regard $C_{\rm red}(\mathbb{G}):=\rho(C(\mathbb{G}))$ as our main object. 
The GNS representation of $h$ is written as $L^2(\mathbb{G})$. 
All dual objects are written with hat (e.g.\ $\hat{\mathbb{G}}$).

Let $F$ be a matrix in $\mathrm{GL}(n,\mathbb{C})$. The \textit{free unitary quantum group} (resp.\ \textit{free orthogonal quantum group}) of $F$ \cite{VW96}\cite{Wa95} is the $C^*$-algebra $C(A_u(F))$ (resp.\ $C(A_o(F))$) defined as the universal unital $C^*$-algebra generated by all the entries of a unitary $n$ by $n$ matrix $u=(u_{i,j})_{i,j}$ satisfying that $F(u_{i,j}^*)_{i,j}F^{-1}$ is unitary (resp.\ $F(u_{i,j}^*)_{i,j}F^{-1}=u$).

In \cite{Is13}, we observed for free quantum groups, there is a nuclear $C^*$-algebra 
\begin{equation*}
\mathcal{C}_l \subset C^*\{C_{\rm red}(\mathbb{G}), \hat{\lambda}(\ell^\infty(\hat{\mathbb{G}}))\}\subset \mathbb{B}(L^2(\mathbb{G}))
\end{equation*}
such that
\begin{itemize}
	\item[(a)] it contains $C_{\rm red}(\mathbb{G})$ and $\mathbb{K}(L^2(\mathbb{G}))$;
	\item[(b)] all commutators of $\mathcal{C}_l$ and $C_{\rm red}(\mathbb{G})^{\rm op}$ are contained in $\mathbb{K}(L^2(\mathbb{G}))$, where $C_{\rm red}(\mathbb{G})^{\rm op}$ acts on $L^2(\mathbb{G})$ canonically.
\end{itemize}
When we see it in the continuous core $L^\infty(\mathbb{G})\rtimes\mathbb{R}\subset\mathbb{B}(L^2(\mathbb{G})\otimes L^2(\mathbb{R}))$ (with respect to the Haar state $h$), it also satisfies:
\begin{itemize}
	\item[(c)] a family of maps $\mathrm{Ad}\Delta_h^{it}$ $(t\in\mathbb{R})$ gives a norm continuous action of $\mathbb{R}$ on $\mathcal{C}_l$;
	\item[(d)] all commutators of $\pi(\mathcal{C}_l)$ and $C_{\rm red}(\mathbb{G})^{\rm op} \otimes_{\rm min}1$ are contained in $\mathbb{K}(L^2(\mathbb{G}))\otimes_{\rm min}\mathbb{B}(L^2(\mathbb{R}))$.
\end{itemize}
Here  $\pi$ means the canonical $*$-homomorphism from $\mathbb{B}(L^2(\mathbb{G}))$ into $\mathbb{B}(L^2(\mathbb{G})\otimes L^2(\mathbb{R}))$ defined by $(\pi(x)\xi)(t):=\Delta_h^{-it}x\Delta_h^{it}\xi(t)$ for $x\in\mathbb{B}(L^2(\mathbb{G}))$, $t\in\mathbb{R}$, and $\xi\in L^2(\mathbb{G})\otimes L^2(\mathbb{R})$.

In the paper, our essential assumptions on quantum groups are these four conditions and we actually treat quantum subgroups at the same time. 
We are only interested in non-amenable von Neumann algebras, since amenable one is not prime. 
So we use the following terminology. 
\begin{Def}\label{class C}\upshape
Let $\hat{\mathbb{G}}$ be a discrete quantum group. We say $\hat{\mathbb{G}}$ is in $\mathcal{C}$ if $L^\infty(\mathbb{G})$ is non-amenable and there exists a discrete quantum group $\hat{\mathbb{H}}$ such that
\begin{itemize}
	\item the quantum group $\hat{\mathbb{G}}$ is a quantum subgroup of $\hat{\mathbb{H}}$;
	\item there exists a nuclear $C^*$-algebra $\mathcal{C}_l  \subset C^*\{C_{\rm red}(\mathbb{H}), \hat{\lambda}(\ell^\infty(\hat{\mathbb{H}}))\}\subset \mathbb{B}(L^2(\mathbb{H}))$ which satisfies conditions from (a) to (d) for $\hat{\mathbb{H}}$. 
\end{itemize}
\end{Def}

\bigskip

In our previous work, we already found following examples (see Subsection 3.2 and the proof of Theorem C in \cite{Is13}).
\begin{Pro}
Let $\mathbb{G}$ be one of the following quantum groups.
\begin{itemize}
	\item[$\rm (i)$] A co-amenable compact quantum group.
	\item[\rm (ii)] The free unitary quantum group $A_u(F)$ for any $F\in\mathrm{GL}(n,\mathbb{C})$.
	\item[\rm (iii)] The free orthogonal quantum group $A_o(F)$ for any $F\in\mathrm{GL}(n,\mathbb{C})$.
	\item[\rm (iv)] The quantum automorphism group $A_{\rm aut}(B, \phi)$ for any finite dimensional $C^*$-algebra $B$ and any faithful state $\phi$ on $B$.
	\item[\rm (v)] The dual of a bi-exact discrete group $\Gamma$.
	\item[\rm (vi)] The dual of a free product $\hat{\mathbb{G}}_1*\cdots *\hat{\mathbb{G}}_n$, where each $\mathbb{G}_i$ is as in from $\rm (i)$ to $\rm (v)$ above.
\end{itemize}
Then the dual $\hat{\mathbb{G}}$ is in $\mathcal{C}$ if $L^\infty(\mathbb{G})$ is non-amenable.
\end{Pro}

Since conditions (a) and (b) above implies bi-exactness \cite[\textrm{Lemmas 3.1.4 and 3.3.1}]{Is13}, which obviously implies condition (AO), we easily deduce the following lemma.
\begin{Lem}\label{C is full}
Let $\hat{\mathbb{G}}$ be in $\mathcal{C}$. Then we have $C^*\{L^\infty(\mathbb{G}),L^\infty(\mathbb{G})'\}\cap \mathbb{K}(L^2(\mathbb{G}))\neq0$. 
\end{Lem}

Let $\mathbb{G}$ be a compact quantum groups and $h$ its Haar state. 
Let $(u_{i,j}^x)_{i,j}^x$ be a fixed basis of the dense Hopf $*$-algebra $\mathcal{A}$ of $C_{\rm red}(\mathbb{G})$. Assume that they are orthogonal in $L^2(\mathbb{G})$. 
In this case, the modular action of the Haar state $h$ satisfies $\sigma_t^h(u_{i,j}^x)=(\lambda_i^x\lambda_j^x)^{it} u_{i,j}^x$ for some scalars $\lambda_i^x$. 
Let $E_h$ be the $h$-preserving conditional expectation from $L^\infty(\mathbb{G})$ onto $L^\infty(\mathbb{G})_h$ which extends to the projection $e_h$ from $L^2(\mathbb{G})$ onto $L^2(L^\infty(\mathbb{G})_h)$. 
Let $\mathcal{A}_h\subset L^\infty(\mathbb{G})_h$ be all the linear spans of $\{u_{i,j}^x\mid \lambda_i^x\lambda_j^x=1 \}$, which is a $*$-algebra, and $C_{\rm red}(\mathbb{G})_h$ be its norm closure. 
For $a\in L^\infty(\mathbb{G})_h$ with the Fourier expansion $a=\sum_{x\in\mathrm{Irred}(\mathbb{G}),i,j} a_{i,j}^xu_{i,j}^x$, 
we have $a=\sigma_t^h(a)=\sum_{x\in\mathrm{Irred}(\mathbb{G}),i,j} a_{i,j}^x (\lambda_i^x\lambda_j^x)^{it}u_{i,j}^x$ in $L^2(\mathbb{G})$ and hence $a_{i,j}^x=0$ if $\lambda_i^x\lambda_j^x\neq 1$. 
This means that $e_h$ is the projection onto the subspace spanned by $\mathcal{A}_h$, and hence we have $E_h(u_{i,j}^x)=u_{i,j}^x$ when $\lambda_i^x\lambda_j^x= 1$, and $E_h(u_{i,j}^x)=0$ when $\lambda_i^x\lambda_j^x\neq 1$. 
Hence $E_h$ is a map from $\mathcal{A}$ onto $\mathcal{A}_h$ which is the identity on $\mathcal{A}_h$. This implies that $\mathcal{A}_h$ is $\sigma$-weakly dense in $L^\infty(\mathbb{G})_h$ and $E_h$ restricts to a conditional expectation from $C_{\rm red}(\mathbb{G})$ onto $C_{\rm red}(\mathbb{G})_h$.

%%%%%%%%%%%%%%%%%%%%%%%%%%%%%%%%%%%%
\subsection{\bf Popa's intertwining techniques}\label{Popa's intertwining techniques}
%%%%%%%%%%%%%%%%%%%%%%%%%%%%%%%%%%%%

In the paper, we use Popa's intertwining techniques for type III subalgebras, although there is no useful equivalent conditions in this case. 
So our definition here is more general than usual one. 

\begin{Def}\label{Popa embed def}\upshape
Let $M$ be a von Neumann algebra, $p$ and $q$ projections in $M$, $A\subset pMp$ and $B\subset qMq$ von Neumann subalgebras. Assume that $B$ is finite or type $\rm III$. We say $A$ \textit{embeds in $B$ inside} $M$ and denote by $A\preceq_M B$ if there exist non-zero projections $e\in A$ and $f\in B$, a unital normal $*$-homomorphism $\theta \colon eAe \rightarrow fBf$, and a partial isometry $v\in M$ such that
\begin{itemize}
\item $vv^*\leq e$ and $v^*v\leq f$,
\item $v\theta(x)=xv$ for any $x\in eAe$.
\end{itemize}
\end{Def}

\bigskip

We first recall  characterizations of this condition for non-finite and semifinite von Neumann algebras (with finite subalgebras). 

\begin{Thm}[\textit{non-finite version, }{\cite{Po01}\cite{Po03}\cite{Ue12}\cite{HV12}}]\label{Popa embed}
Let $M,p,q,A,$ and $B$ be as in the definition above and let $E_{B}$ be a faithful normal conditional expectation from $qMq$ onto $B$. Assume that $B$ is finite with a trace $\tau_B$. 
Then the following conditions are equivalent.
\begin{itemize}
	\item[$\rm (i)$] We have $A\preceq_M B$.
	\item[$\rm (ii)$] There exists no sequence $(w_n)_n$ of unitaries in $A$ such that $\|E_{B}(b^*w_n a)\|_{2,\tau_B}\rightarrow 0$ for any $a,b\in pMq$.
%	\item[$\rm (iii)$] There exists a non-zero positive element $d_0\in \langle M, \tilde{B}\rangle$ such that $dp\rho(q)=d$, $\mathrm{Tr}_{\langle M, \tilde{B}\rangle}(d_0)<\infty$, and $K^A_{d_0}\not\ni 0$. 
%Here $K^A_{d_0}$ is the $\sigma$-weak closure of $\mathrm{co}\{wd_0w^*\mid w\in\mathcal{U}(A)\}$.
%	\item[$\rm (iv)$] There exists a non-zero positive element $d\in \langle M,\tilde{B}\rangle\cap A'$ with $dp\rho(q)=d$ and $\mathrm{Tr}_{\langle M,\tilde{B}\rangle}(d)<\infty$.
	\item[$\rm (iii)$] There exists a non-zero $A$-$B$-submodule $H$ of $pL^2(M)q$ with $\dim_{(B,\tau_B)}H<\infty$.
\end{itemize}
\end{Thm}

\begin{Thm}[\textit{semifinite version, }{\cite{CH08}\cite{HR10}}]\label{popa embed2}
Let $M$ be a semifinite von Neumann algebra with a faithful normal semifinite trace $\mathrm{Tr}$, and $B\subset M$ be a von Neumann subalgebra with $\mathrm{Tr}_{B}:=\mathrm{Tr}|_{B}$ semifinite. Denote by $E_{B}$  the unique $\mathrm{Tr}$-preserving conditional expectation from $M$ onto $B$. 
Let $p$ be a $\mathrm{Tr}$-finite projection in $M$ and $A\subset pMp$ a von Neumann subalgebra. Then the following conditions are equivalent.
\begin{itemize}
	\item[$\rm (i)$]There exists a non-zero projection $p\in B$ with $\mathrm{Tr}_{B}(p)<\infty$ such that $A\preceq_{rMr}qBq$, where $r:=p\vee q$.
	\item[$\rm (ii)$]There exists no sequence $(w_n)_n$ in unitaries of $A$ such that $\|E_{B}(b^{\ast}w_n a)\|_{2,\mathrm{Tr}_B}\rightarrow 0$ for any $a,b\in pM$.
\end{itemize}
We use the same symbol $A\preceq_{M}B$ if one of these conditions holds.
\end{Thm}
\begin{Rem}\upshape
By the proof, when $A\preceq_{M}B$ for semifinite $B\subset M$, for any increasing net $(q_i)_i\subset B$ of $\mathrm{Tr}_B$-finite projections, we can find some $q_i$ such that $A\preceq_{rMr} q_iBq_i$ where $r:=p\vee q_i$. 
So we can choose such a $q_i$ from any semifinite subalgebra of $B$ on which $\mathrm{Tr}_B$ is semifinite. 
\end{Rem}

Since we mainly study continuous cores in the paper, the theorem for semifinite algebras are important for us. 
Theorem \ref{popa embed2} is a generalization of the statement (ii) in Theorem \ref{Popa embed} and this formulation is, for example, useful in the next subsection. 
However, to study our factorization properties, the statement (iii) is also important since the proof of Ozawa and Popa for prime factorization result relied on them. 
Hence here we give corresponding statements in the semifinite setting. 
We include sketches of proofs, since we use similar strategies later. 
See \cite[\textrm{Theorem F.12}]{BO08} for the details. 
Recall that any opposite von Neumann algebra $M^{\rm op}$ has a canonical right action on $L^2(M)$, and we write its element as $a^{\rm op}\in M^{\rm op}$ for $a\in M$.
%In the theorem below, for a tracial von Neumann algebra $M$ and its element $a\in M$, we use the notation $a^{\rm op}$ for a canonical right action of $a$ on $L^2(M)$ given by $a^{\rm op} \widehat{x}=\widehat{xa}$. 

\begin{Pro}\label{Popa embed2}
Let $M,B,\mathrm{Tr},p,A$, and $E_B$ be as in the previous theorem. 
Let $q\in B$ be a non-zero projection with $\mathrm{Tr}_{B}(q)<\infty$ and write as $z_q$ the central support projection of $q$ in $B$. Let $e_B$ be the Jones projection of $B\subset M$ and let $\mathrm{Tr}_{\langle M, B\rangle}$ be the semifinite trace on $\langle M, B\rangle$ given by $Me_BM\ni xe_By\mapsto \mathrm{Tr}(xy)$. 
Then the following conditions are equivalent. 
\begin{itemize}
	\item[$\rm (i)$] We have $A\preceq_{rMr}qBq$ for $r:=p\vee q$.
	\item[$\rm (ii)$] There exists a non-zero positive element $d_0\in \langle M, B\rangle$ such that $d_0pz_q^{\rm op}=d_0$, $\mathrm{Tr}_{\langle M, B\rangle}(d_0)<\infty$, and $K^A_{d_0}\not\ni 0$. Here $K^A_{d_0}$ is the $\sigma$-weak closure of $\mathrm{co}\{wd_0w^*\mid w\in\mathcal{U}(A)\}$.% and $z_q^{\rm op}$ is the canonical right action of $z_q$ on $L^2(M)$ given by $z_q^{\rm op}\widehat{x}:=\widehat{xz_q}$.
	\item[$\rm (iii)$] There exists a non-zero positive element $d\in \langle M, B\rangle\cap A'$ such that $dpz_q^{\rm op}=d$ and $\mathrm{Tr}_{\langle M, B\rangle}(d)<\infty$.
	\item[$\rm (iv)$] There exists a non-zero $pAp$-$qBq$-submodule $H$ of $pL^{2}(M,\mathrm{Tr})q$ with $\dim_{(qBq,\tau_{q})}H<\infty$, where $\tau_{q}:=\mathrm{Tr}_B(q\cdot q)/\mathrm{Tr}_B(q)$.
\end{itemize}
We use the same symbol $A\preceq_{M}B$ if one of these conditions holds.
\end{Pro}
\begin{proof}
The equivalence of (i) and (iv) follows from Theorem \ref{Popa embed}, since we canonically have $pL^{2}(M,\mathrm{Tr})q=pL^2(rMr,\mathrm{Tr}(r\cdot r))q$. We show implications $\rm (i)\Rightarrow (ii) \Rightarrow (iii) \Rightarrow (iv)$. 

Suppose condition (i). 
Then by the theorem above, we can find $\delta>0$ and a finite subset $\mathcal{F}\subset pMq$ such that $\|E_B(y^*wx)\|_{2, \tau_q}>\delta$ for any $x,y\in\mathcal{F}$ and $w\in\mathcal{U}(A)$, where $\tau_q:=\mathrm{Tr}(q\cdot q)/\mathrm{Tr}(q)$. 
Put $d_0:=\sum_{x\in \mathcal{F}}xe_Bx^*\in \langle M,B \rangle$. It obviously satisfies $d_0=d_0p$ and $\mathrm{Tr}_{\langle M, B\rangle}(d_0)<\infty$. 
We also have $d_0z_q^{\rm op}=d_0$ since $z_q^{\rm op}e_B=z_qe_B$. 
To see the condition on $K_{d_0}^A$, we calculate for any $w\in\mathcal{U(A)}$
\begin{equation*}
\sum_{x\in\mathcal{F}}\langle w^*d_0w \hat{x}|\hat{x}\rangle
=\sum_{x,y\in\mathcal{F}}\langle e_B y^*w \hat{x}|y^*w\hat{x}\rangle
=\sum_{x,y\in\mathcal{F}}\| e_B y^*w \hat{x}\|_{2, \tau_q}
=\sum_{x,y\in\mathcal{F}}\| E_B( y^*w x)\|_{2, \tau_q}>\delta.
\end{equation*}
Hence we get $K_{d_0}^A\not\ni 0$ and (ii) holds.

Next suppose condition (ii). We regard $K_{d_0}^A\subset \langle M,B\rangle$ as an $L^2$-norm bounded, convex, weak closed (and hence $L^2$-norm closed) subset of $L^2(\langle M, B\rangle,\mathrm{Tr}_{\langle M, B\rangle})$. Then we can find the circumcenter $d$ of $K_{d_0}^A$, which is non-zero by assumption. By the uniqueness of the circumcenter, we have $wdw^*=d$ and hence $d$ is contained in $\langle M,B\rangle\cap A'$. 
Since $d\in K_{d_0}^A$, we have $\mathrm{Tr}_{\langle M, B\rangle}(d)<\infty$ by the normality of $\mathrm{Tr}_{\langle M, B\rangle}$. Thus condition (iii) holds.

Suppose condition (iii). Take a non-zero spectral projection $f$ of $d$ such that $\mathrm{Tr}_{\langle M, B\rangle}(f)<\infty$. 
Since $f$ is also contained in $\langle M,B\rangle \cap A'$ with $f=fpz_q^{\rm op}$, the closed subspace $H:=fq^{\rm op}L^2(M)$, which is nonzero since $f\leq z_q^{\rm op}$, has a $pAp$-$qBq$-bimodule structure as a submodule of $L^2(M)$. 
Now the dimension of $H$ with respect to $(qBq, \tau_q)$ is smaller than $\mathrm{Tr}_{\langle M, B\rangle}(f)/\mathrm{Tr}_{\langle M, B\rangle}(q)$ and hence $H$ has a finite dimension. We get condition (iv).
\end{proof}

\begin{Cor}\label{Popa embed3}
Keep the setting in the previous proposition. Then $A\preceq_{M}B$ if and only if  there exists a non-zero positive element $d\in p\langle M, B\rangle p\cap A'$ such that $\mathrm{Tr}_{\langle M, B\rangle}(d)<\infty$. 
Also $A\preceq_M B$ if and only if $eAe\preceq_{rMr} fBf$ for any (some) projections $e\in A$ and $f\in B$ with $\mathrm{Tr}_B(f)<\infty$ and with central supports $1_A$ and $1_B$ respectively. 
\end{Cor}
\begin{proof}
For the first half, the only if direction is trivial from the previous theorem. Let $(q_i)\subset B$ be an increasing net of $\mathrm{Tr}_B$-finite projections converging to $1_B$. Let $d$ be in the statement. Then there is some $i$ such that $\tilde{d}:=dz_{q_i}^{\rm op}\neq0$. 
Then $\tilde{d}$ satisfies $\mathrm{Tr}_{\langle M, B\rangle}(\tilde{d})<\infty$ and $\tilde{d}=\tilde{d}z_{q_i}^{\rm op}p$. Since $z_{q_i}^{\rm op}\in \mathcal{Z}(\langle M, B\rangle)\subset \langle M, B\rangle\cap A'$, we have $\tilde{d}\in \langle M, B\rangle\cap A'$. 

For the second half, suppose $A\preceq_M B$ and take $d$ in the first statement. If $de=0$, then $dz_A(e)=d$ is also zero since $d\in A'$. So $de=dez_f^{\rm op}$ is non-zero and we have $eAe\preceq_{rMr} fBf$.
\end{proof}

\begin{Cor}\label{Popa embed4}
Keep the setting in the previous proposition and assume that $A$ is of type $\rm II$. 
Then the following condition is also equivalent to $A\preceq_{rMr} qBq$. 
\begin{itemize}
	\item[$\rm (v)$] There exist non-zero projections $e\in A$ and $f\in qBq$, and a partial isometry $V\in \langle M, B\rangle$ such that $V^*V\in eA'$ and $VV^*=fe_B$, where $e_B$ is the Jones projection for $B\subset M$.
\end{itemize}
In the case, $e\in A$ is taken from any type $\rm II$ subalgebra of $A$.
\end{Cor}
\begin{proof} 
Before the proof, we observe that there is an extended center valued trace $T$ on $z_q^{\rm op}\langle M, B\rangle$ such that $T(qe_B)=z_q^{\rm op}$ (for extended center valued traces, see the final part of \cite[\textrm{Subsection V.2}]{Ta1}). Note that $z_q^{\rm op}$ is the central support of $q^{\rm op}$ in $B^{\rm op}$ and hence is contained in $\mathcal{Z}(\langle M, B\rangle)$. 

Since the central support of $qe_B$ in $\langle M, B\rangle$ is $z_q^{\rm op}$, we have an isomorphism 
\begin{equation*}
\pi \colon z_q^{\rm op}\mathcal{Z}(\langle M, B\rangle)=(z_q\mathcal{Z}(B))^{\rm op}\ni (z_qa)^{\rm op}\mapsto qe_B(z_qa)^{\rm op}=qe_Ba\in qe_B\mathcal{Z}(B).
\end{equation*} 
Put $\phi:=\mathrm{Tr}_{\langle M, B\rangle}\circ\pi$. 
Then $\phi$ is a faithful normal positive functional on $z_q^{\rm op}\mathcal{Z}(\langle M, B\rangle)$ and hence there is an extended center valued trace $T$ on $z_q^{\rm op}\langle M, B\rangle$ defined by the equation $\mathrm{Tr}_{\langle M, B\rangle}(xa)=\phi(T(x)a)$ for all $x\in z_q^{\rm op}\langle M, B\rangle^+$ and $a\in z_q^{\rm op}\mathcal{Z}(\langle M, B\rangle)^+$. 
It holds that $T(qe_B)=z_q^{\rm op}$. 

Suppose condition (iii) in the previous proposition. Taking a spectral projection, we assume $d$ is a projection. 
Let $T$ be the extended center valued trace on $z_q^{\rm op}\langle M, B\rangle$ constructed above. 
Since $\mathrm{Tr}_{\langle M, B\rangle}(d)<\infty$, it is a finite projection and so it takes finite values almost everywhere as a function on the spectrum of $\mathcal{Z}(\langle M, B\rangle)$ (e.g.\ \cite[\textrm{Proposition V.2.35}]{Ta1}). Hence there exists a projection $z\in \mathcal{Z}(\langle M, B\rangle)$ such that $T(zd)\leq nz_q^{\rm op}$ for some $n\in\mathbb{N}$. 
Replacing $d$ with $zd$, we may assume $T(d)\leq nz_q^{\rm op}$. 
Since $A$ is of type II, there exist mutually orthogonal and equivalent projections $e_i$ such that $1_A=\sum_{i=1}^ne_i$. 
Then it is easy to show that $0\neq T(de_1)\leq z_q^{\rm op}$. Write $e:=e_1$. 
Then $T(de)\leq z_q^{\rm op}$ implies $de\prec qe_B$. 
There is a partial isometry $V\in z_q^{\rm op}\langle M, B\rangle$ such that $V^*V=de$ and $VV^*\leq qe_B$. 
Writing $VV^*=fe_B$ for some $f\in qBq$, we are done. 
Note that $e_i$ above are taken from any type II subalgebra of $A$ and hence the final assertion also holds. 

Next suppose condition (v). 
Since $VeAeV^*\subset fe_B\langle M, B\rangle fe_B=fBfe_B\simeq fBf$ and $V^*V$ commutes with $eAe$, we can define a unital normal $*$-homomorphism $\theta$ from $eAe$ into $fBf$ by $\theta(x)=VxV^*$. 
Put $\xi:=V^*\hat{f}\in L^2(eMf)\subset L^2(rMr)$. It holds that $x\xi=\xi\theta(x)$ for any $x\in eAe$. Take a polar decomposition of $\xi$ in $L^2(rMr)$ and denote its pole by $v\in rMr$. Then we have $v=ev=vf$ and $xv=v\theta(x)$ for any $x\in eAe$. 
This means $A\preceq_{rMr} qBq$. 
\end{proof}

We next consider the case that the subalgebra $B$ is of type III. In the case, equivalent conditions above no longer hold, but still there are some fundamental properties. 
For simplicity, from now on, we assume that subalgebras $A$ and $B$ are of type $\rm II_1$ or III. 
We write as $1_A$ and $1_B$ the units of $A$ and $B$.

\begin{Pro}
The following conditions are equivalent.
\begin{itemize}
	\item[$\rm (i)$] We have $A\preceq_MB$.
	\item[$\rm (ii)$] There are a nonzero normal $*$-homomorphism $\pi\colon A\mathbin{\bar{\otimes}} \mathbb{C}e_{1,1}\rightarrow B\mathbin{\bar{\otimes}}\mathbb{M}_n$ for some $n\in \mathbb{N}$ and a nonzero partial isometry $w\in (1_A\otimes e_{1,1})(M\mathbin{\bar{\otimes}}\mathbb{M}_{n})$ such that $w\pi(x)=(x\otimes e_{1,1})w$ for all $x\in A$, where $(e_{i,j})_{i,j}$ is a fixed matrix unit in $\mathbb{M}_n$. 
\end{itemize}
\end{Pro}
\begin{proof}
The case that $B$ is of type $\rm II_1$ was already discussed in \cite[\textrm{Proposition 3.1}]{Ue12}. So we assume that $B$ is of type III.  

Suppose $A\preceq_M B$ and take $e,f,v$ and $\theta$ as in the definition. 
If $A$ is of type III, then $e\sim z_A(e)$ in $A$ with a partial isometry $u\in A$ such that $uu^*=e$ and $u^*u=z_A(e)$, where $z_A(e)$ is the central support projection of $e$ in $A$.  
The composite map 
\begin{equation*}
A\xrightarrow{\times z_A(e)} Az_A(e)\xrightarrow{\ \mathrm{Ad}u\ } eAe\xrightarrow{\hspace{0.5em}\theta\hspace{0.5em}} fBf
\end{equation*}
and the partial isometry $u^*v$ satisfy condition (ii) (for $n=1$). 
If $A$ is of type $\rm II_1$, then the same argument as in \cite[\textrm{Proposition 3.1}]{Ue12} works.

Suppose next condition (ii). 
The central support of $\pi(1_A)$ in $B\mathbin{\bar{\otimes}} \mathbb{M}_n$ is of the form $z_B\otimes 1$ for some central projection $z_B\in B$. 
Then since central supports of $\pi(1_A)$ and $z_B\otimes e_{1,1}$ are same in the type III algebra $B\mathbin{\bar{\otimes}}\mathbb{M}_n$, we have $\pi(1_A)\sim z_B\otimes e_{1,1}$ with a partial isometry $u\in B\mathbin{\bar{\otimes}} \mathbb{M}_n$ such that $u^*u=\pi(1_A)$ and $uu^*=z_B\otimes e_{1,1}$. 
Then the composite map 
\begin{equation*}
A\mathbin{\bar{\otimes}} \mathbb{C}e_{1,1}\xrightarrow{\hspace{0.3em}\pi\hspace{0.3em}} \pi(1_A)(B\mathbin{\bar{\otimes}} \mathbb{M}_n)\pi(1_A)\xrightarrow{\mathrm{Ad}u} (z_B\otimes e_{1,1})(B\mathbin{\bar{\otimes}} \mathbb{M}_n)(z_B\otimes e_{1,1})= Bz_B\mathbin{\bar{\otimes}} \mathbb{C}e_{1,1}
\end{equation*}
and the partial isometry $wu^*\in (1_A\otimes e_{1,1})(M\mathbin{\bar{\otimes}}\mathbb{M}_n)(z_B\otimes e_{1,1})\subset M\mathbin{\bar{\otimes}} \mathbb{C}e_{1,1}$ work by identifying $M\mathbin{\bar{\otimes}} \mathbb{C}e_{1,1}\simeq M$.
\end{proof}

From the characterization, it is easy to deduce the following properties.

\begin{Cor}
The following statements are true.
\begin{itemize}
	\item[$\rm (i)$] If $pAp\preceq_M B$ for some $p\in A$ or $A'\cap M$, then $A\preceq_M B$.
	\item[$\rm (ii)$] If $A\preceq_M qBq$ for some $q\in B$ or $B'\cap M$, then $A\preceq_M B$.
	\item[$\rm (iii)$] If $A\preceq_M B$, then $D\preceq_M B$ for any type $\rm II_1$ or $\rm III$ unital subalgebra $D\subset A$.
\end{itemize}
\end{Cor}

The next lemma is a special case of \cite[\textrm{Lemma 3.5}]{Va08}, but it treats type III subalgebras. 

\begin{Lem}\label{relative commutant}
Assume that $M$ and $B$ are of the form $M=M_1\mathbin{\bar{\otimes}} M_2$ and $B=M_1$ for von Neumann algebras $M_1$ and $M_2$. 
Assume that $M_2$ is of type $\rm III$ and $A'\cap 1_AM1_A$ is of type $\rm II_1$ or $\rm III$. If $A\preceq_M M_1$, then $M_2\preceq_M A'\cap 1_AM1_A$.
\end{Lem}
\begin{proof}
Take $e\in A$, $f\in M_1$, $\theta$, and $v\in eMf$ as in the definition of $A\preceq_MM_1$. Write $p=vv^*$ and $q:=v^*v$. 
We have an inclusion $v^*p(eAe)pv\subset q\theta(eAe)$ as subalgebras of $qMq$. 
Taking relative commutants in $qMq$, we have 
\begin{eqnarray*}
q(\theta(eAe)'\cap fMf)q=(q\theta(eAe))'\cap qMq \subset (v^*eAev)' \cap qMq 
=v^*p((eAe)'\cap eMe)pv.
\end{eqnarray*}
This implies $q(\theta(eAe)'\cap fMf)q\preceq_{M}p((eAe)'\cap eMe)p$ with the partial isometry $v$. 
Note that $\theta(eAe)'\cap fMf$ is of the form $(\theta(eAe)'\cap fM_1f)\mathbin{\bar{\otimes}} M_2$ and hence is of type III. 
So by the previous lemma, we have $\theta(eAe)'\cap fMf\preceq_{M}(eAe)'\cap eMe=e(A'\cap 1_AM1_A)e$ and hence $fM_2\preceq_{M}e(A'\cap 1_AM1_A)e$, since $fM_2\subset \theta(eAe)'\cap fMf$. 
Again by the lemma, we have $M_2\preceq_M A'\cap 1_AM1_A$.
\end{proof}

%%%%%%%%%%%%%%%%%%%%%%%%%%%%%%%%%%%%
\subsection{\bf Approximately containment in continuous cores}\label{appro in core}
%%%%%%%%%%%%%%%%%%%%%%%%%%%%%%%%%%%%

In the subsection, we introduce a notion of approximately containment formulated by Vaes \cite[\textrm{Section 2}]{Va10}. Our notion here is a slightly generalized but essentially the same one. 
Since we follow the same strategy as in \cite[\textrm{Section 2}]{Va10}, we basically omit the proofs. 

Let $\mathbb{G}$ be a compact quantum group and $h$ its Haar state.
During the subsection, we consider following two cases at the same time.
\begin{itemize}
	\item {\bf Case 1.} Our target is the continuous core $L^\infty(\mathbb{G})\rtimes_{\sigma^h}\mathbb{R}$ and its canonical trace, denoted by $M$ and Tr.
	\item {\bf Case 2.} The quantum group $\mathbb{G}$ is of Kac type and the dual $\hat{\mathbb{G}}$ acts on a semifinite tracial von Neumann algebra $(N,\mathrm{Tr}_N)$ as a trace preserving action. Our target is the crossed product von Neumann algebra $\hat{\mathbb{G}}\ltimes N$ and its canonical trace, which is also denoted by $M$ and Tr.
\end{itemize}
For a discrete quantum subgroup $\hat{\mathbb{H}}\subset \hat{\mathbb{G}}$, we write $M_{\hat{\mathbb{H}}}:=L^\infty(\mathbb{H})\rtimes_{\sigma^h}\mathbb{R}$ or $\hat{\mathbb{H}}\ltimes N$. 
For $\mathcal{E},\mathcal{F}\subset \mathrm{Irred}(\mathbb{G})$ and $x,y\in\mathrm{Irred}(\mathbb{G})$, we use the notation
\begin{eqnarray*}
\mathcal{E}\mathcal{F}&:=&\{z\in\mathrm{Irred}(\mathbb{G})\mid z\in e\otimes f\textrm{ for some }e\in\mathcal{E}, f\in\mathcal{F}\},\\
x\mathcal{E}y&:=&\{z\in\mathrm{Irred}(\mathbb{G})\mid z\in x\otimes e\otimes y\textrm{ for some }e\in\mathcal{E}\}.
\end{eqnarray*}
Let $\mathcal{S}$ be a family of discrete quantum subgroups of $\hat{\mathbb{G}}$. 
We say a subset $\mathcal{F}\subset \mathrm{Irred}(\mathbb{G})$ is \textit{small relative to} $\mathcal{S}$ if it is contained in a finite union of subsets of the form $x\mathrm{Irred}(\mathbb{H})y$ for some $x,y\in\mathrm{Irred}(\mathbb{G})$ and  $\hat{\mathbb{H}}\in\mathcal{S}$. 
For any subset $\mathcal{F}\subset \mathrm{Irred}(\mathbb{G})$, we write the orthogonal projection from $L^2(\mathbb{G})\otimes L^2(\mathbb{R})$ onto $L^2(\mathcal{F})\otimes L^2(\mathbb{R})$ (or $L^2(\mathbb{G})\otimes L^2(N)$ onto $L^2(\mathcal{F})\otimes L^2(N)$) as $P_\mathcal{F}$, where $L^2(\mathcal{F})$ is the closed subspace spanned by all $u_{i,j}^x$ for $x\in\mathcal{F}$. 
For a subgroup $\hat{\mathbb{H}}$, the restriction of $P_{\mathrm{Irred}(\mathbb{H})}$ on $M$ is the trace preserving conditional expectation onto $M_{\hat{\mathbb{H}}}$. 
We write $\mathcal{N}_{\mathrm{Tr}}:=\{a\in M\mid \mathrm{Tr}(a^*a)<\infty\}$ and note that any element $a\in \mathcal{N}_{\mathrm{Tr}}$ has a Fourier expansion in $L^2(M,\mathrm{Tr})$ written as $a=\sum_{x\in\mathrm{Irred}(\mathbb{G}),i,j}u_{i,j}^xa_{i,j}^x$ for $a_{i,j}^x=E(u_{i,j}^{x*}a)$, where we take $(u_{i,j}^x)$ as an orthogonal family in $L^2(\mathbb{G})$ and $E$ is the Tr-preserving conditional expectation onto  $L\mathbb{R}$ or $N$.

\begin{Def}\upshape
Let $\mathcal{V}$ be a norm bounded subset in $\mathcal{N}_{\mathrm{Tr}}\subset M$. 
We say $\mathcal{V}$ is \textit{approximately contained in} $M^\mathcal{S}$, and denote by $\mathcal{V}\subset_{\rm approx} M^\mathcal{S}$, if for any $\epsilon>0$, there is $\mathcal{F}\subset \mathrm{Irred}(\mathbb{G})$ which is small relative to $\mathcal{S}$ such that 
\begin{equation*}
\|b-P_\mathcal{F}(b)\|_2=\|P_{\mathcal{F}^c}(b)\|_2<\epsilon
\end{equation*}
for all $b\in\mathcal{V}$.
\end{Def}

We also use the notation 
$B\subset_{\rm approx} M^\mathcal{S}$ when $(B)_1\subset_{\rm approx} M^\mathcal{S}$ for a subalgebra $B\subset M$, 
and $\mathcal{V}\subset_{\rm approx} L^\infty(\mathbb{H})\rtimes\mathbb{R}$ or $\hat{\mathbb{H}}\ltimes N$ when $\mathcal{S}=\{\hat{\mathbb{H}}\}$.

We start with a simple lemma.

\begin{Lem}\label{nice ineq}
For any $\mathcal{F}\subset\mathrm{Irred}(\mathbb{G})$, $u^a\in \mathrm{span}\{u_{i,j}^a\mid i,j=1,\ldots,n_a \}$ for $a=x,y\in\mathrm{Irred}(\mathbb{G})$, and $b\in \mathcal{N}_{\mathrm{Tr}}$, we have
\begin{eqnarray*}
\|P_\mathcal{F}(u^{x*} b u^y) \|_2\leq \|u^{x*}\| \|P_{x\mathcal{F}\bar{y}}(b)\|_2\|u^y\|.
\end{eqnarray*}
\end{Lem}
\begin{proof}
We first assume $y=\epsilon$, the trivial corepresentation. 
Write $b=\sum_{z\in\mathrm{Irred}(\mathbb{G}), p,q}u_{p,q}^zb_{p,q}^z$ for some $b_{p,q}^z\in L\mathbb{R}$ (or $N$). 
For $z\in\mathrm{Irred}(\mathbb{G})$, $(\bar{x}\otimes z)\cap \mathcal{F}\neq\emptyset$ is equivalent to $z\in x\mathcal{F}$.
Hence we have
\begin{eqnarray*}
\|P_\mathcal{F}(u^{x*} b) \|_2
&=&\|\sum_{z\in\mathrm{Irred}(\mathbb{G}), p,q}P_\mathcal{F}(u^{x*} u_{p,q}^zb_{p,q}^z) \|_2\\
&=&\|\sum_{z\in x\mathcal{F}, p,q}P_\mathcal{F}(u^{x*} u_{p,q}^zb_{p,q}^z) \|_2\\
&\leq&\|u^{x*}\|\|\sum_{z\in x\mathcal{F}, p,q}u_{p,q}^zb_{p,q}^z \|_2 =\|u^{x*}\|\|P_{x\mathcal{F}}(b)\|_2.
\end{eqnarray*}
By the same manner, we also have $\|P_\mathcal{F}( bu^y) \|_2\leq \|P_{\mathcal{F}\bar{y}}(b) \|_2\|u^y\|$. 
Hence 
\begin{equation*}
\|P_\mathcal{F}(u^{x*} bu^y) \|_2\leq \|u^{x*}\|\|P_{x\mathcal{F}}(bu^y) \|_2\leq \|u^{x*}\|\|P_{x\mathcal{F}\bar{y}}(b) \|_2\|u^{y}\|.
\end{equation*} 
\end{proof}

By the lemma, we can prove the following lemma, which corresponds to \cite[\textrm{Lemma 2.3}]{Va10}.

\begin{Lem}\label{2.3}
Let $\mathcal{V}\subset \mathcal{N}_{\mathrm{Tr}}$ be a norm bounded subset. If $\mathcal{V}\subset_{\rm approx} M^\mathcal{S}$, then $x\mathcal{V}y\subset_{\rm approx} M^\mathcal{S}$ for any $x,y\in\mathcal{N}_{\mathrm{Tr}}$. 

Let $p\in M$ be a $\mathrm{Tr}$-finite projection and $(v_\lambda)_\lambda$ be a bounded net in $pMp$. If $\|P_\mathcal{F}(v_\lambda)\|_2\rightarrow 0$ for all $\mathcal{F}$ which is small relative to $\mathcal{S}$, then 
$\|P_{\mathrm{Irred}(\mathbb{H})}(b^*v_\lambda a)\|_2\rightarrow 0$ for all $a,b\in pM$ and all $\hat{\mathbb{H}}\in\mathcal{S}$.
\end{Lem}

To continue our argument, we need one more assumption which is an opposite phenomena to the second statement in the last lemma. 
More precisely we need the following condition:
\begin{itemize}
	\item For any $\hat{\mathbb{H}}\in\mathcal{S}$, $\mathrm{Tr}$-finite projection $p\in M$, and any net $(w_\lambda)_\lambda$ in $\mathcal{U}(pMp)$, 
if $\|P_{\mathrm{Irred}(\mathbb{H})}(b^*w_\lambda a)\|_2$ converges to 0 for all $a,b\in pM$, then $\|P_{x\mathrm{Irred}(\mathbb{H})\bar{y}}(w_\lambda)\|_2$ also converges to 0 for all $x,y\in\mathrm{Irred}(\mathbb{G})$.
\end{itemize}
When we treat a discrete group, since the inequality in Lemma \ref{nice ineq} becomes the equality (because $P_\mathcal{F}(\lambda_x^*b\lambda_y)=P_{x\mathcal{F}y^{-1}}(b)$), this condition trivially holds. 
However, in the quantum situation, this equality is no longer true and hence we have to assume that our target $\hat{\mathbb{G}}$ and $\mathcal{S}$ satisfy this condition. 
Fortunately our main target, direct product quantum groups, always satisfy the condition.
\begin{Lem}
Assume either that $\rm (i)$ $\hat{\mathbb{G}}$ is a group or $\rm (ii)$ $\hat{\mathbb{G}}$ is a direct product quantum group $\hat{\mathbb{G}}_1\times\cdots \times\hat{\mathbb{G}}_m$ and $\mathcal{S}$ consists of subgroups generated by some of $\hat{\mathbb{G}}_i$ for $i=1,\ldots,m$. 
Then the condition above holds.
\end{Lem}
\begin{proof}
The assumption (i) was already mentioned just before the lemma. So we assume (ii). 
Take any $\hat{\mathbb{H}}\in\mathcal{S}$ which is generated by some $\hat{\mathbb{G}}_i$. 
Exchanging their indices and writing $\hat{\mathbb{G}}=\hat{\mathbb{H}}_1\times \hat{\mathbb{H}}_2$, where $\hat{\mathbb{H}}_1:=\hat{\mathbb{G}}_1\times\cdots\times\hat{\mathbb{G}}_n$ and $\hat{\mathbb{H}}_2:=\hat{\mathbb{G}}_{n+1}\times\cdots\times\hat{\mathbb{G}}_m$ for some $n$, 
we may assume $\hat{\mathbb{H}}=\{\epsilon_1\} \times\hat{\mathbb{H}}_2$. Here $\epsilon_1$ is the unit of $\hat{\mathbb{H}}_1$. 
In the case, for any $x,y\in\mathrm{Irred}(\mathbb{G})$ there is a finite subset $\mathcal{F}\subset \mathrm{Irred}(\mathbb{H}_1)$ such that $x\mathrm{Irred}(\mathbb{H})y\subset \mathcal{F}\times \mathrm{Irred}(\mathbb{H}_2)$. 
Since $P_{\mathcal{F}\times \mathrm{Irred}(\mathbb{H}_2)}=\sum_{z\in\mathcal{F}}P_{\{z\}\times\mathrm{Irred}(\mathbb{H}_2)}$, we may assume $x\in \mathrm{Irred}(\mathbb{H}_1)$ and $y$ is trivial. 

Take a Fourier expansion of $w_\lambda$ along $\hat{\mathbb{H}}_1$, namely, decompose $w_\lambda=\sum_{z\in\mathrm{Irred}(\mathbb{H}_1),k,l}u^z_{k,l} (w_\lambda)^z_{k,l}$ for  $(w_\lambda)^z_{k,l}=P_{\mathrm{Irred}(\mathbb{H}_2)}(u_{k,l}^{z*}w_\lambda)\in M_{\hat{\mathbb{H}}_2}$, 
where $(u_{k,l}^z)_{k,l}^z$ is taken as an orthogonal system. 
Then for any $u_{i,j}^x$ we have
\begin{eqnarray*}
0\leftarrow \|u_{i,j}^x\|_2^{-2}\|P_{\mathrm{Irred}(\mathbb{H})}(u^{x*}_{i,j}w_\lambda)\|_2
=\|u_{i,j}^x\|_2^{-2}\|\sum_{y,k,l}h(u^{x*}_{i,j}u_{k,l}^y) (w_\lambda)_{k,l}^y\|_2
=\|(w_\lambda)_{i,j}^x\|_2,
\end{eqnarray*}
and hence
\begin{eqnarray*}
\|P_{\{x\}\times\mathrm{Irred}(\mathbb{H}_2)}(\omega_\lambda)\|_2
=\|\sum_{i,j} u_{i,j}^x (w_\lambda)_{i,j}^x \|_2
\leq \sum_{i,j}\|(w_\lambda)_{i,j}^x \|_2\rightarrow0.
\end{eqnarray*}
\end{proof}

From now on, we assume that $\mathbb{G}$ and $\mathcal{S}$ satisfy the condition above. 
Then we can follow all the proofs in \cite[\textrm{Section 2}]{Va10} and get the following four statements.
\begin{Lem}\label{2.4}
Let $p$ be a $\mathrm{Tr}$-finite projection in $M$ and $B\subset pMp$ be a von Neumann subalgebra generated by a group of unitaries $\mathcal{G}\subset \mathcal{U}(B)$. The following statements are equivalent.
\begin{itemize}
	\item For every $\hat{\mathbb{H}}\in\mathcal{S}$, we have $B\not\preceq_{M} M_{\hat{\mathbb{H}}}$.
	\item There exists a net of unitaries $(w_i)$ in $\mathcal{G}$ such that $\|P_\mathcal{F}(w_i)\|_2\rightarrow 0$ for every subset $\mathcal{F}\subset \mathrm{Irred}(\mathbb{G})$ which is small relative to $\mathcal{S}$.
\end{itemize}
\end{Lem}
\begin{Lem}\label{2.5}
Let $p$ be a $\mathrm{Tr}$-finite projection in $M$ and $B\subset pMp$ be a von Neumann subalgebra. The following statements are equivalent.
\begin{itemize}
	\item[1.] There exists an $\hat{\mathbb{H}}\in \mathcal{S}$ such that $B \preceq_{M}M_{\hat{\mathbb{H}}}$.
	\item[2.] There exists a nonzero projection $q\in B'\cap pMp$ such that $(Bq)_1\subset_{\rm approx} M^\mathcal{S}$.
\end{itemize}
Also the following two statements are equivalent.
\begin{itemize}
	\item[a.] For every nonzero projection $q\in B'\cap pMp$, there exists an $\hat{\mathbb{H}}\in\mathcal{S}$ such that $Bq \preceq_{M}M_{\hat{\mathbb{H}}}$.
	\item[b.] We have $(B)_1\subset_{\rm approx} M^\mathcal{S}$.
\end{itemize}
\end{Lem}
\begin{Pro}\label{2.6}
The set of projections
\begin{equation*}
\mathcal{P}:=\{q_0\in B'\cap pMp\mid (Bq_0)_1\subset_{\rm approx}M^\mathcal{S}\}
\end{equation*}
attains its maximum in a unique projection $q\in\mathcal{P}$. This projection belongs to $\mathcal{Z}(\mathcal{N}_{pMp}(B)'')$.
\end{Pro}
\begin{Lem}\label{approx direct}\label{2.7}
Let $\mathcal{S}_1$, $\mathcal{S}_2$ and $\mathcal{S}$ be families of discrete quantum subgroups of $\hat{\mathbb{G}}$. 
Assume that for any $\mathcal{F}_i$ which is small relative to $\mathcal{S}_i$, $\mathcal{F}_1\cap\mathcal{F}_2$ is small relative to $\mathcal{S}$. 
If $(B)_1\subset_{\rm approx}M^{\mathcal{S}_i}$ for $i=1,2$, then $(B)_1\subset_{\rm approx}M^\mathcal{S}$.
\end{Lem}

\begin{Rem}\upshape\label{2.8}
As mentioned in \cite[\textrm{Lemma 2.7}]{Va10}, when $\hat{\mathbb{G}}$ is a discrete group $\Gamma$, we can put $\mathcal{S}:=\{\Sigma_1\cap g\Sigma_2g^{-1} \mid \Sigma_i\in\mathcal{S}_i \textrm{ for $i=1,2$ and } g\in\Gamma \}$. 
For our main target, direct product quantum groups, we can put $\mathcal{S}:=\{\hat{\mathbb{H}}_1\cap \hat{\mathbb{H}}_2 \mid \hat{\mathbb{H}}_i\in\mathcal{S}_i \textrm{ for $i=1,2$}\}$.
\end{Rem}

In the proof of main theorems, we often use the following lemma, which is an easy consequence of lemmas above.
\begin{Lem}\label{2.9}
Let $p$ be a $\mathrm{Tr}$-finite projection in $M$ and $B\subset pMp$ be a von Neumann subalgebra. 
Assume that $B'\cap pMp$ is a factor. Then for any $\mathcal{S}$, $(B)_1\subset_{\rm approx} M^\mathcal{S}$ if and only if there exists an $\hat{\mathbb{H}}\in \mathcal{S}$ such that $B \preceq_{M}M_{\hat{\mathbb{H}}}$.
\end{Lem}
\begin{proof}
The only if direction is trivial. 
Assume $B \preceq_{M}M_{\hat{\mathbb{H}}}$ for some $\hat{\mathbb{H}}\in \mathcal{S}$. Then by Lemma \ref{2.5}, there is a nonzero projection $q_0\in B'\cap pMp$ such that $(q_0B)_1\subset_{\rm approx} M^\mathcal{S}$. 
Take the maximum projection $q\in B'\cap pMp$ in Proposition \ref{2.6} such that $(qB)_1\subset_{\rm approx} M^\mathcal{S}$. 
Since $q$ is non-zero and is contained in $B'\cap pMp \cap \mathcal{Z}(\mathcal{N}_{pMp}(B)'')\subset \mathcal{Z}(B'\cap pMp)=\mathbb{C}$, $q=1$ and we get $(B)_1\subset_{\rm approx} M^\mathcal{S}$.
\end{proof}

%%%%%%%%%%%%%%%%%%%%%%%%%%%%%%%%%%%%
\section{\bf Prime factorization results for type III factors}
%%%%%%%%%%%%%%%%%%%%%%%%%%%%%%%%%%%%

In the section, we prove the main theorem. As we mentioned in Introduction, we begin our work by generalizing condition (AO) to the tensor product setting. 
We then study locations of subalgebras and intertwiners inside cores of tensor products by Popa's intertwining method. 
We obtain factorization results for cores first, and then deduce factorization results for the original algebras.

%%%%%%%%%%%%%%%%%%%%%%%%%%%%%%%%%%%%
\subsection{\bf Condition (AO) for tensor product algebras}\label{AO tensor}
%%%%%%%%%%%%%%%%%%%%%%%%%%%%%%%%%%%%

As we mentioned, a discrete quantum group in $\mathcal{C}$ is bi-exact, and hence the associated von Neumann algebra satisfies condition (AO). 
In the subsection, we study some appropriate conditions on direct products of quantum groups in $\mathcal{C}$. 
Such an observation was first given in \cite{OP03} for direct products and then generalized in \cite[\textrm{Section 15}]{BO08} with the notion of relative bi-exactness. 
Our approach here is close to the first one. 
In the subsection, we use the notation 
\begin{eqnarray*}
\mathbb{B}_{\mathbb{G}}:=\mathbb{B}(L^2(\mathbb{G})),\  \mathbb{K}_{\mathbb{G}}:=\mathbb{K}(L^2(\mathbb{G})), \quad
\mathbb{B}_{\mathbb{R}}:=\mathbb{B}(L^2(\mathbb{R})),\  \mathbb{K}_{\mathbb{R}}:=\mathbb{K}(L^2(\mathbb{R})),
\end{eqnarray*}
for any compact quantum group $\mathbb{G}$ and the real numbers $\mathbb{R}$. 

Let $\hat{\mathbb{G}}_i$ $(i=1,\ldots,m)$ be discrete quantum groups in $\mathcal{C}$. By definition there are $\hat{\mathbb{H}}_i$ containing $\hat{\mathbb{G}}_i$ with nuclear $C^*$-algebras $\mathcal{C}_l^i$. 
For simplicity, we first assume $\hat{\mathbb{G}}_i=\hat{\mathbb{H}}_i$. We will get the same conclusions in the general case (Lemma \ref{general case}).  
We use the notation
\begin{eqnarray*}
&&\mathcal{C}_l:=\mathcal{C}_l^1\otimes_{\rm min} \cdots \otimes_{\rm min} \mathcal{C}_l^m, \quad
\hat{\mathbb{G}}:=\hat{\mathbb{G}}_1\times \cdots \times \hat{\mathbb{G}}_m,\\
&&\hat{\mathbb{G}}_i':=\hat{\mathbb{G}}_1\times \cdots\times \hat{\mathbb{G}}_{i-1}\times \{\epsilon_i\}\times\hat{\mathbb{G}}_{i+1}\times\cdots \times \hat{\mathbb{G}}_m,
\end{eqnarray*}
where $\epsilon_i$ is the unit of $\hat{\mathbb{G}}_i$. 
Note that $\mathcal{C}_l$ is also nuclear and  contains $C_{\rm red}(\mathbb{G})$. 
Consider a multiplication map
\begin{eqnarray*}
\nu\colon \mathcal{C}_l\otimes_{\rm alg} C_{\rm red}(\mathbb{G})^{\rm op}\ni a\otimes b^{\rm op} \mapsto ab^{\rm op}\in \mathbb{B}(L^2(\mathbb{G})),
\end{eqnarray*}
This is a well defined linear map but not a $*$-homomorphism in general. 
In fact, $\nu$ is a $*$-homomorphism if and only if commutators $[\nu(\mathcal{C}_l),\nu(C_{\rm red}(\mathbb{G})^{\rm op})]$ are zero. 
If there is a non-unital $C^*$-algebra $J\subset \mathbb{B}(L^2(\mathbb{G}))$ which contains $[\nu(\mathcal{C}_l),\nu(C_{\rm red}(\mathbb{G})^{\rm op})]$ and whose multiplier algebra $M(J)$ contains $\mathrm{ran}\nu$, 
then exchanging the range of $\nu$ from $\mathbb{B}(L^2(\mathbb{G}))$ with $M(J)/J$, $\nu$ becomes a $*$-homomorphism. 
In our situation, we can find an appropriate $J$ as follows.

\begin{Lem}\label{lemma J}
Denote
\begin{eqnarray*}
\mathcal{K}_i:=\mathbb{B}_{\mathbb{G}_1}\otimes_{\rm min}\cdots\otimes_{\rm min}\mathbb{B}_{\mathbb{G}_{i-1}}\otimes_{\rm min} \mathbb{K}_{\mathbb{G}_i}\otimes_{\rm min}\mathbb{B}_{\mathbb{G}_{i+1}}\otimes_{\rm min}\cdots\otimes_{\rm min}\mathbb{B}_{\mathbb{G}_n}.
\end{eqnarray*}
Put $J:=\sum_i\mathcal{K}_i$. 
Then $J$ is a $C^*$-algebra containing $[\nu(\mathcal{C}_l),\nu(C_{\rm red}(\mathbb{G})^{\rm op})]$ and the multiplier algebra $M(J)$ contains $\mathrm{ran}\nu$.
\end{Lem}
\begin{proof}
Since each $\mathcal{K}_i$ is an ideal in $\mathbb{B}_{\mathbb{G}_1}\otimes_{\rm min}\cdots\otimes_{\rm min}\mathbb{B}_{\mathbb{G}_n}$, $J$ is a $C^*$-algebra. 
Since $\mathrm{ran}\nu$ is contained in $\mathbb{B}_{\mathbb{G}_1}\otimes_{\rm min}\cdots\otimes_{\rm min}\mathbb{B}_{\mathbb{G}_n}$, it is contained in $M(J)$. 
To show $[\nu(\mathcal{C}_l),\nu(C_{\rm red}(\mathbb{G})^{\rm op})] \subset J$, we have only to check $[\nu(\mathcal{C}_l^i),\nu(C_{\rm red}(\mathbb{G}_j)^{\rm op})]\subset J$ for any $i,j$, because $J$ is an ideal in $M(J)$. 
The commutators are zero when $i\neq j$, and are contained in $\mathcal{K}_i$ when $i=j$.
\end{proof}

We keep $J$ in the lemma. Then we get a $*$-homomorphism $\nu$ from $\mathcal{C}_l\otimes_{\rm alg} C_{\rm red}(\mathbb{G})^{\rm op}$ into $M(J)/J$, which is bounded with respect to the max tensor product norm. Since $\mathcal{C}_l$ is nuclear, the max norm coincides with the minimal tensor norm. 
Finally restricting the map on $C_{\rm red}(\mathbb{G})\otimes_{\rm alg} C_{\rm red}(\mathbb{G})^{\rm op}$, we get the following proposition which is an analogue of condition (AO) on tensor product algebras.

\begin{Pro}\label{AO for tensor algebras}
The $C^*$-algebra $C_{\rm red}(\mathbb{G})$ is exact and the multiplication map 
\begin{equation*}
C_{\rm red}(\mathbb{G})\otimes_{\rm alg} C_{\rm red}(\mathbb{G})^{\rm op}\ni a\otimes b^{\rm op}\mapsto ab^{\rm op} \in M(J)/J
\end{equation*}
is bounded with respect to the minimal tensor norm.
\end{Pro}

We next investigate a similar property on continuous cores. 
Recall that commutants of continuous cores are of the form 
\begin{eqnarray*}
&&L^\infty(\mathbb{G})\rtimes\mathbb{R}=W^*\{ \pi(L^\infty(\mathbb{G})),\ 1\otimes \lambda_t \ (t\in \mathbb{R})  \},\\
&&(L^\infty(\mathbb{G})\rtimes\mathbb{R})'=(L^\infty(\mathbb{G})\rtimes\mathbb{R})^{\rm op}=W^*\{ L^\infty(\mathbb{G})^{\rm op}\mathbin{\bar{\otimes}} 1,\ \Delta^{it}\otimes \rho_t  \ (t\in \mathbb{R})  \}.
\end{eqnarray*}
where $\pi$ is the canonical $*$-homomorphism into $\mathbb{B}(L^2(\mathbb{G})\otimes L^2(\mathbb{R}))$, $\Delta$ is the modular operator of the Haar state of $\mathbb{G}$, and $\rho_t$ is the right regular representation of $\mathbb{R}$. 
We go along a similar line to above by exchanging $C_{\rm red}(\mathbb{G})$ with $C_{\rm red}(\mathbb{G})\rtimes_{\rm r} \mathbb{R}$, the reduced norm continuous crossed product, which is a dense subalgebra in the core. 

Keep the notation above and let us first consider a multiplication map
\begin{eqnarray*}
\mu\colon \mathcal{C}_l\otimes_{\rm alg} C_{\rm red}(\mathbb{G})^{\rm op}\ni a\otimes b \mapsto \pi(a) (b\otimes 1)\in \mathbb{B}(L^2(\mathbb{G})\otimes L^2(\mathbb{R})).
\end{eqnarray*}
In (the proof of) \cite[\textrm{Proposition 3.2.3}]{Is12_2}, we verified that $M(\mathbb{K}_{\mathbb{G}}\otimes_{\rm min} \mathbb{B}_{\mathbb{R}})$ contains the $C^*$-algebra $D$ which is generated by
\begin{itemize}
	\item the image of $\mu$; 
	\item $1\otimes\lambda_t$, $\Delta^{it}\otimes\rho_{t}$ $(t\in\mathbb{R})$;
	\item $\int_\mathbb{R} f(s)(1\otimes\lambda_s)\cdot ds$, $\int_\mathbb{R} f(s)(\Delta^{is}\otimes\rho_{s})\cdot ds$ $(f\in L^1(\mathbb{R}))$
\end{itemize}
(for the fact, we do not need any assumption on $\mathbb{G}$). 
In our situation, since $[\mu(\mathcal{C}_l),\mu(C_{\rm red}(\mathbb{G})^{\rm op})]$ is not contained in $\mathbb{K}_{\mathbb{G}}\otimes_{\rm min} \mathbb{B}_{\mathbb{R}}$, the algebra $M(\mathbb{K}_{\mathbb{G}}\otimes_{\rm min} \mathbb{B}_{\mathbb{R}})$ should be exchanged by $M(\tilde{J})$ for some appropriate $\tilde{J}$. 

\begin{Lem}
Put $\tilde{J}:=J\otimes_{\rm min} \mathbb{B}_{\mathbb{R}}=\sum_i\mathcal{K}_i\otimes_{\rm min} \mathbb{B}_{\mathbb{R}}$. 
Then $\tilde{J}$ is a $C^*$-algebra containing $[\mu(\mathcal{C}_l),\mu(C_{\rm red}(\mathbb{G})^{\rm op})]\subset \tilde{J}$ and the multiplier algebra $M(\tilde{J})$ contains $D$.
\end{Lem}
\begin{proof}
Since each $\mathcal{K}_i\otimes_{\rm min} \mathbb{B}_{\mathbb{R}}$ is an ideal in $\mathbb{B}_{\mathbb{G}_1}\otimes_{\rm min}\cdots\otimes_{\rm min}\mathbb{B}_{\mathbb{G}_n}\otimes_{\rm min} \mathbb{B}_{\mathbb{R}}$, $\tilde{J}$ is a $C^*$-algebra. 
Obviously $1\otimes\lambda_t$ and $\Delta^{it}\otimes\rho_{t}$ $(t\in\mathbb{R})$, and $\mu(C_{\rm red}(\mathbb{G})^{\rm op})$ are contained in $\mathbb{B}_{\mathbb{G}_1}\otimes_{\rm min}\cdots\otimes_{\rm min}\mathbb{B}_{\mathbb{G}_n}\otimes_{\rm min} \mathbb{B}_{\mathbb{R}}$ and hence in $M(\tilde{J})$. 
Let $(p^i_j)_j$ be an increasing net of central projections in $\hat{\lambda}(c_0(\hat{\mathbb{G}}_i)) \subset \mathbb{K}_{\mathbb{G}_i}$ converging to $1$ strongly, which automatically commute with $\Delta^{it}$ $(t\in \mathbb{R})$ and $\hat{\lambda}(\ell^\infty(\hat{\mathbb{G}}_i))$. 
Then $\int_\mathbb{R} f(s)(1\otimes\lambda_s)\cdot ds$ and $\int_\mathbb{R} f(s)(\Delta^{is}\otimes\rho_{s})\cdot ds$ $(f\in L^1(\mathbb{R}))$, and $\mu(\ell^\infty(\hat{\mathbb{G}}))$ commute with $1_{\mathbb{G}_i'}p_j^i\otimes 1_\mathbb{R}\in \mathcal{K}_i\otimes_{\rm min}\mathbb{B}_\mathbb{R}$. 
So they are contained in $M(\mathcal{K}_i\otimes_{\rm min}\mathbb{B}_\mathbb{R})$ for all $i$ and hence in $M(\tilde{J})$. 
Let $(u_{k,l}^x)$ be a basis of the dense Hopf $*$-algebra of $C_{\rm red}(\mathbb{G})$, which is orthogonal in $L^2(\mathbb{G})$. 
Then since $\sigma_t^h(u_{k,l}^x)=(\lambda_{k,l}^x)^{it}u_{k,l}^x$ $(t\in \mathbb{R})$ for some $\lambda_{k,l}^x>0$, $\pi(u_{k,l}^x)$ is of the form $u_{k,l}^x\otimes f_{k,l}^x$, where $f_{k,l}^x\in L^\infty(\mathbb{R})$ is given by $f_{k,l}^x(t)=(\lambda_{k,l}^x)^{-it}$. 
This is contained in $M(\tilde{J})$ and hence $\mu(\mathcal{C}_l)$ is contained in $M(\tilde{J})$. 

To prove $[\mu(\mathcal{C}_l),\mu(C_{\rm red}(\mathbb{G})^{\rm op})]\subset \tilde{J}$, it suffice to show $[\mu(\mathcal{C}_l^i),\mu(C_{\rm red}(\mathbb{G}_j)^{\rm op})]\subset \tilde{J}$ for any $i,j$. 
The commutators are zero when $i\neq j$, and are contained in $\mathcal{K}_i\otimes_{\rm min} \mathbb{B}_{\mathbb{R}}$ when $i=j$. 
\end{proof}

We keep $\tilde{J}$ in the lemma. 
Exchange the range of $\mu$ with $M(\tilde{J})/\tilde{J}$ and get a $*$-homomorphism. 
It is bounded with respect to the minimal tensor norm since $\mathcal{C}_l$ is nuclear. 
Then consider $(\mathbb{R}\times\mathbb{R}$)-actions on $\mathcal{C}_l\otimes_{\rm min} C_{\rm red}(\mathbb{G})^{\rm op}$ and $M(\tilde{J})/\tilde{J}$ given by $(s,t)\mapsto \mathrm{Ad}(\Delta^{is}\otimes\Delta^{it}) $ and $(s,t)\mapsto \mathrm{Ad}([1\otimes\lambda_s][\Delta^{it}\otimes \rho_t])$ respectively.  
It is easy to verify that $\mu$ is $(\mathbb{R}\times\mathbb{R})$-equivariant. Since $M(\tilde{J})$ contains $D$ and $\mathbb{R}\times\mathbb{R}$ is amenable, we get the following map
\begin{equation*}
(\mathcal{C}_l\rtimes_{\rm r}\mathbb{R})\otimes_{\rm min} (C_{\rm red}(\mathbb{G})\rtimes_{\rm r}\mathbb{R})^{\rm op}\simeq(\mathcal{C}_l\otimes_{\rm min} C_{\rm red}(\mathbb{G})^{\rm op})\rtimes_{\rm r}(\mathbb{R}\times\mathbb{R})\rightarrow M(\tilde{J})/\tilde{J}.
\end{equation*}
Restricting the map, we get the following proposition, which is also an analogue of condition (AO) on tensor products. 

\begin{Pro}\label{AO for tensor cores}
The $C^*$-algebra $C_{\rm red}(\mathbb{G})\rtimes_{\rm r}\mathbb{R}$ is exact and the multiplication map 
\begin{equation*}
(C_{\rm red}(\mathbb{G})\rtimes_{\rm r}\mathbb{R})\otimes_{\rm alg} (C_{\rm red}(\mathbb{G})\rtimes_{\rm r}\mathbb{R})^{\rm op}\ni a\otimes b^{\rm op}\mapsto ab^{\rm op} \in M(\tilde{J})/\tilde{J}
\end{equation*}
is bounded with respect to the minimal tensor norm.
\end{Pro}

Recall that we assumed $\hat{\mathbb{G}}_i=\hat{\mathbb{H}}_i$ at the first stage in the subsection. 
In the general case, cutting by $e_{\mathbb{G}_i}$, which is the projection from $L^2(\mathbb{H}_i)$ onto $L^2(\mathbb{G}_i)$, we can easily deduce the following lemma (e.g.\ \cite[\textrm{Lemma 3.3.1}]{Is13}). 

\begin{Lem}\label{general case}
Propositions \ref{AO for tensor algebras} and \ref{AO for tensor cores} are true for general $\hat{\mathbb{G}}_i$ in $\mathcal{C}$.
\end{Lem}

\begin{Rem}\upshape\label{AO for centralizers}
In the end of Subsection \ref{CQG}, we observed that the Haar state preserving conditional expectation $E_h$ onto $L^\infty(\mathbb{G})_h$ gives a conditional expectation from $C_{\rm red}(\mathbb{G})$ onto $C_{\rm red}(\mathbb{G})_h$. 
By the same argument above, we can also exchange all objects of $\hat{\mathbb{G}}_i$ in Proposition \ref{AO for tensor algebras} with $C_{\rm red}(\mathbb{G}_i)_h$, $L^\infty(\mathbb{G}_i)_h$ and $L^2(L^\infty(\mathbb{G}_i)_h)$. 
Hence by Ozawa--Popa's method, we can prove prime factorization results for $L^\infty(\mathbb{G}_1)_{h_1}\mathbin{\bar{\otimes}}\cdots\mathbin{\bar{\otimes}} L^\infty(\mathbb{G}_n)_{h_n}$ if each tensor component is a non-amenable $\rm II_1$ factor.
\end{Rem}

%%%%%%%%%%%%%%%%%%%%%%%%%%%%%%%%%%%%
\subsection{\bf Location of subalgebras}
%%%%%%%%%%%%%%%%%%%%%%%%%%%%%%%%%%%%

In this subsection, we give a key observation for our factorization results. Our condition (AO) phenomena is used only to prove the following proposition. 
It is a generalization of \cite[\textrm{Theorem C}]{Is12_2} and its origin is \cite[\rm Theorem\ 4.6]{Oz04} (see also \cite[\rm Theorem\ 5.3.3]{Is12_1}). 

Let $\hat{\mathbb{G}}_i$ $(i=1,\ldots,m)$ be discrete quantum groups in $\mathcal{C}$ and  $h_i$ be Haar states of $\mathbb{G}_i$. 
Write $h:=h_1\otimes \cdots\otimes h_n$. 
We use the notation in the previous section such as $\hat{\mathbb{G}}$, $\hat{\mathbb{G}}_i'$ and $\tilde{J}$. 

\begin{Pro}\label{location core}
Let $p$ be a $\mathrm{Tr}$-finite projection in $C_h(L^\infty(\mathbb{G}))(=:C_h)$, where $\mathrm{Tr}$ is the canonical trace on $C_h$, and $N\subset pC_hp$ a von Neumann subalgebra. 
Then we have either one of the following statements:
\begin{itemize}
	\item[$\rm (i)$] The relative commutant $N'\cap pC_hp$ is amenable.
	\item[$\rm (ii)$] We have $N\preceq_{C_h} L^\infty(\mathbb{G}_i')\rtimes \mathbb{R}$ for some $i$.
\end{itemize}
\end{Pro}
\begin{proof}
We give only a sketch of the proof. 
By \cite[\textrm{Proposition 5.2.4}]{Is12_1}, we may assume that $N$ is amenable. 
Suppose by contradiction that $N\not\preceq_{rC_hr} q(L^\infty(\mathbb{G}_i')\rtimes \mathbb{R}) q$ for any $i$ and any $\mathrm{Tr}$-finite projection $q\in L^\infty(\mathbb{G}_i')\rtimes \mathbb{R}$ and $r:=p\vee q$. 

Define a proper conditional expectation $\Psi_N\colon \mathbb{B}(L^2(C_h)) \rightarrow N'$. 
For any $i$, we will show $\Psi_N(e_{\mathbb{G}_i} \otimes 1_{\mathbb{G}_i'}\otimes 1_\mathbb{R})=0$, where $e_{\mathbb{G}_i}$ is the projection from $L^2(\mathbb{G}_i)$ onto $\mathbb{C}\hat{1}_{\mathbb{G}_i}$. 
This implies $\mathbb{K}(L^2(\mathbb{G}_i))\otimes_{\rm min} \mathbb{B}(L^2(\mathbb{G}_i')\otimes L^2(\mathbb{R})) \subset \ker\Psi_N$ and hence  $\tilde{J}\subset \ker\Psi_N$. 

Let $q\in L^\infty(\mathbb{G}_i')\rtimes \mathbb{R}$ be any projection with $\mathrm{Tr}(q)<\infty$ and $z_q$ be the central support of $q$ in $L^\infty(\mathbb{G}_i')\rtimes\mathbb{R}$. Write $s:=z_q^{\rm op}\Psi_N(e_{\mathbb{G}_i} \otimes 1_{\mathbb{G}_i'}\otimes 1_\mathbb{R})z_q^{\rm op}$. 
Then $s$ is contained in $N'\cap ((L^\infty(\mathbb{G}_i')\rtimes\mathbb{R})^{\rm op})'$ and satisfies $s=sz_q^{\rm op}p$. 
Since $N\not\preceq_{rC_hr} q(L^\infty(\mathbb{G}_i')\rtimes \mathbb{R}) q$, by Proposition \ref{Popa embed2}, $\mathrm{Tr}_i(s)$ is zero or infinite, where $\mathrm{Tr}_i$ is the canonical trace on the basic construction of $L^\infty(\mathbb{G}_i')\rtimes \mathbb{R}\subset C_h$. 
This actually has a finite value because 
\begin{eqnarray*}
\mathrm{Tr}_i(s) &\leq& \mathrm{Tr}_i(\Psi_N(e_{\mathbb{G}_i} \otimes 1_{\mathbb{G}_i'}\otimes 1_\mathbb{R}))\\
&\leq& \mathrm{Tr}_i(p(e_{\mathbb{G}_i} \otimes 1_{\mathbb{G}_i'}\otimes 1_\mathbb{R})p)\\
&\leq& \mathrm{Tr}_i((e_{\mathbb{G}_i} \otimes 1_{\mathbb{G}_i'}\otimes 1_\mathbb{R})E_i(p))\\
&=&\mathrm{Tr}(E_i(p))=\mathrm{Tr}(p)<\infty,
\end{eqnarray*}
where $E_i$ is the Tr-preserving conditional expectation from $C_h$ onto $L^\infty(\mathbb{G}_i')\rtimes \mathbb{R}$. 
Thus we get $\mathrm{Tr}_i(s)=0$ and this means $\Psi_N(e_{\mathbb{G}_i} \otimes 1_{\mathbb{G}_i'}\otimes 1_\mathbb{R})=0$. 

Now considering the composite map
\begin{equation*}
(C_{\rm red}(\mathbb{G})\rtimes_{\rm r}\mathbb{R})\otimes_{\rm min} (C_{\rm red}(\mathbb{G})\rtimes_{\rm r}\mathbb{R})^{\rm op} \xrightarrow{\mu} M(\tilde{J})/\tilde{J} \xrightarrow{\Psi} N',
\end{equation*}
where $\mu$ is as in Proposition \ref{AO for tensor cores}, 
we can follow the same method as in \cite[\rm Theorem\ 4.6]{Oz04} (or \cite[\rm Theorem\ 5.3.3]{Is12_1}).
\end{proof}

%%%%%%%%%%%%%%%%%%%%%%%%%%%%%%%%%%%%
\subsection{\bf Intertwiners inside type III algebras and continuous cores}
%%%%%%%%%%%%%%%%%%%%%%%%%%%%%%%%%%%%

In the work of Ozawa and Popa on prime factorizations, they first found an intertwiner between tensor components of both sides and then constructed a unitary element which moves tensor components from one side to the other. 
In the subsection we study corresponding statements for type III factors. 
We recall the statement on $\rm II_1$ factors and then prove a similar one on type $\rm III$ factors.

\begin{Lem}[{\cite[\textrm{Proposition 12}]{OP03}}]\label{hariawase1}
Let $M_i$ and $N_i$ be $\rm II_1$ factors with $M_1\mathbin{\bar{\otimes}} M_2= N_1\mathbin{\bar{\otimes}} N_2$ $(=:M)$. Assume $N_1\preceq_M M_1$. 
Then there is a unitary element $u\in M$ and a decomposition $M\simeq M_1^t\mathbin{\bar{\otimes}} M_2^{1/t}$ for some $t>0$  such that $uN_1u^*\subset M_1^t$. 
\end{Lem}

\begin{Lem}\label{hariawase2}
Let $M_i$ and $N_i$ be factors with $M_1\mathbin{\bar{\otimes}} M_2= N_1\mathbin{\bar{\otimes}} N_2(=:M)$. Assume that $M_2$ and $N_2$ are type $\rm III$ factors and $N_1\preceq_M M_1$. 
Then there is a partial isometry $u\in M$ such that $u^*u\in N_1$, $uu^*\in M_1$, and $uN_1u^*\subset uu^*M_1uu^*$. 
If moreover $M_1$ and $N_1$ are also type $\rm III$ factors, the element $u$ can be taken as a unitary element. 
\end{Lem}
\begin{proof}
By assumption, there are $e\in N_1$, $f\in M_1$, $v\in M$, and $\theta\colon eN_1e\rightarrow fM_1f$. 
Since $vv^*\in N_1'\cap eMe=eN_2$ and $eN_2$ is a type III factor, we have $vv^*\sim e$ with a partial isometry $u\in eN_2$ such that $uu^*=e$ and $u^*u=vv^*$. Replacing $v$ with $u^*v$, we may assume that $e=vv^*$. 
Since $v^*v\in \theta(eAe)'\cap fMf= (\theta(eAe)'\cap fM_1f)\mathbin{\bar{\otimes}} M_2$ and $(\theta(eAe)'\cap fM_1f)\mathbin{\bar{\otimes}} M_2$ is of type III, $v^*v\sim z\otimes 1_{M_2}$ for some $z\in \mathcal{Z}(\theta(eAe)'\cap fM_1f)$ with a partial isometry $u\in (\theta(eAe)'\cap fM_1f)\mathbin{\bar{\otimes}} M_2$ such that $uu^*=z\otimes 1_{M_2}$ and $u^*u=v^*v$. 
Replacing $\theta$ and $v$ with $(z\otimes 1_{M_2})\theta(x)$ for $x\in eN_1e$ and $vu^*$, we may assume $f=v^*v$. 
So the first statement holds. 

We next assume $M_1$ and $N_1$ are of type $\rm III$. 
Then since $e\sim 1_{N_1}$ in $N_1$ and $f\sim 1_{M_1}$ in $M_1$, replacing $v$, we can assume $vv^*=v^*v=1$.
\end{proof}

%%%%%%%%%%%%%%%%%%%%%%%%%%%%%%%%%%%%
\subsection{\bf Proof of Theorem \ref{A}}
%%%%%%%%%%%%%%%%%%%%%%%%%%%%%%%%%%%%

We start with two simple lemmas. 

\begin{Lem}
Let $M$ be a von Neumann algebra and $N\subset M$ a subfactor with a faithful normal conditional expectation $E_N$. Let $\phi$ be a faithful normal state on $M$ with $\phi\circ E_N=\phi$. If $M=N\vee (N'\cap M)$, then we have an isomorphism $N\mathbin{\bar{\otimes}} (N'\cap M)\ni a\otimes b\mapsto ab\in M$.
\end{Lem}
\begin{proof}
See the proof of \cite[\textrm{Lemma XIV.2.5}]{Ta3}.
\end{proof}

\begin{Lem}\label{relative commutant in cores}
Let $M_i$ be von Neumann algebras and $\phi_i$ faithful normal states on $M_i$. 
If $M_1\cap (M_1)_{\phi_1}'=\mathbb{C}$, then $C_{\phi_1\otimes \phi_2}(M_1\mathbin{\bar{\otimes}} M_2)\cap (M_1)_{\phi_1}'=C_{\phi_2}(M_2)$. 
If moreover $M_1$ is a $\rm III_1$ factor, then $C_{\phi_1\otimes \phi_2}(M_1\mathbin{\bar{\otimes}} M_2)\cap C_{\phi_1}(M_1)'=(M_2)_{\phi_2}$.
\end{Lem}
\begin{proof}
Write $N:=C_{\phi_1\otimes \phi_2}(M_1\mathbin{\bar{\otimes}} M_2)$. 
Since $N$ is contained in $M_1\mathbin{\bar{\otimes}} M_2 \mathbin{\bar{\otimes}} \mathbb{B}(L^2(\mathbb{R)})$, 
$N\cap (M_1)_{\phi_1}'$ is contained in $M_1\mathbin{\bar{\otimes}} M_2 \mathbin{\bar{\otimes}} \mathbb{B}(L^2(\mathbb{R)})\cap ((M_1)_{\phi_1}\mathbin{\bar{\otimes}} \mathbb{C}\mathbin{\bar{\otimes}} \mathbb{C})'=\mathbb{C}\mathbin{\bar{\otimes}} M_2 \mathbin{\bar{\otimes}} \mathbb{B}(L^2(\mathbb{R)})$. 
Hence it is contained in $\mathbb{C}\mathbin{\bar{\otimes}} C_{\phi_2}(M_2)$ and the first statement holds. 
We then have 
\begin{equation*}
N\cap C_{\phi_1}(M_1)' \subset \mathbb{C}\mathbin{\bar{\otimes}} C_{\phi_2}(M_2)\cap (\mathbb{C}\mathbin{\bar{\otimes}}\mathbb{C}\mathbin{\bar{\otimes}} L\mathbb{R})' =\mathbb{C}\mathbin{\bar{\otimes}} (M_2)_{\phi_2}\mathbin{\bar{\otimes}} L\mathbb{R}
\end{equation*}
(use \cite[Proposition 2.4]{HR10} if necessary). 
%From the previous lemma, $N\cap C_{\phi_1}(M_1)'$ is contained in $\mathbb{C}\mathbin{\bar{\otimes}} (M_2)_{\phi_2}\mathbin{\bar{\otimes}} L\mathbb{R}$. 
If $M_1$ is a $\rm III_1$ factor, exchanging the first and the second tensor components, we have
\begin{equation*} 
((M_2)_{\phi_2}\mathbin{\bar{\otimes}} \mathbb{C}\mathbin{\bar{\otimes}} L\mathbb{R})\cap (\mathbb{C}\mathbin{\bar{\otimes}} C_{\phi_1}(M_1)')\subset (M_2)_{\phi_2} \mathbin{\bar{\otimes}} (C_{\phi_1}(M_1) \cap C_{\phi_1}(M_1)')=(M_2)_{\phi_2} \mathbin{\bar{\otimes}} \mathbb{C}.
\end{equation*}
\end{proof}

By the first lemma, we identify $M$ as $N\mathbin{\bar{\otimes}} (N'\cap M)$ if $N\subset M$ satisfy the assumption of the lemma. 

Let $\mathbb{G}_i$ and $N_j$ be as in the first statement in Theorem \ref{A}. 
For simplicity, we use the notation $M_i:=L^\infty(\mathbb{G}_i)$, $M:=M_1\mathbin{\bar{\otimes}}\cdots\mathbin{\bar{\otimes}} M_m$, $N:=N_1\mathbin{\bar{\otimes}} \cdots\mathbin{\bar{\otimes}} N_n$, $M_X:=\mathbin{\bar{\otimes}}_{i\in X}M_i$ ($M_\emptyset:=\mathbb{C}$) and $N_Y:=\mathbin{\bar{\otimes}}_{j\in Y}N_j$ ($N_\emptyset :=\mathbb{C}$) for any subsets $X\subset \{1,\ldots,m\}$ and $Y\subset \{1,\ldots,n\}$. 
Since the inclusion $N\subset M$ is with expectation, we have an inclusion $C(N)\subset C(M)$ for some continuous cores. 
Since each $N_i$ has an almost periodic state, there exist almost periodic weights $\phi_i$ such that $(N_i)_{\phi_i}$ is non-amenable and $\phi_i$ is semifinite on $(N_i)_{\phi_i}$. 
Write $\phi:=\phi_1\otimes\cdots \otimes\phi_n$. 
We fix an inclusion 
\begin{eqnarray*}
C_\phi(N) \simeq C(N)\subset C(M)\simeq C_h(M),
\end{eqnarray*}
which preserves canonical traces. We denote this trace by Tr. 
Since the inclusion does not preserve $L\mathbb{R}$, we write $L\mathbb{R}$ inside $C_\phi(N)$ and $C_h(M)$ as $L\mathbb{R}_N$ and $L\mathbb{R}_M$ respectively. 
Identifying $M_X=M_X\mathbin{\bar{\otimes}} 1_{M_{X^c}}$ and $N_Y=N_Y\mathbin{\bar{\otimes}} 1_{N_{Y^c}}$, we often write $C_h(M_X)$, $C_\phi(N_Y)$, $(M_X)_h$, and $(N_Y)_\phi$ (instead of $C_{h_X}(M_X)$ for example, where $h_X:=\otimes_{i\in X}h_i$). 

\begin{Lem}\label{reduction lemma}
For any subset $Y\subset \{1,\ldots,n\}$ with $|Y^c|\leq m$ and any $\mathrm{Tr}$-finite projection $p\in L\mathbb{R}_N$, there are $X\subset \{1,\ldots,m\}$ and  $\phi_i$-finite projections $p_i\in (N_i)_{\phi_i}$ such that $|X^c|=|Y^c|$ and $qC_\phi(N_Y)q\preceq_{C(M)} C_h(M_X)$, where $q:=p_1\otimes \cdots \otimes p_n \otimes p$.
\end{Lem}
\begin{proof}
Let $z_i\in \mathcal{Z}((N_i)_{\phi_i})$ be projections such that $z_i(N_i)_{\phi_i}$ has no amenable summand. 
Since $(N_i)_{\phi_i}$ are semifinite, there are $\phi_i$-finite projections $p_i\in(N_i)_{\phi_i}$ such that $p_iz_i\neq 0$. Put $\tilde{p}_i:=p_iz_i$, $\tilde{p}:=\tilde{p}_1\otimes \cdots\otimes\tilde{p}_n$,  $\tilde{N}_i:=\tilde{p}_iN_i\tilde{p}_i$, 
 and $\tilde{N}:=\tilde{N}_1\mathbin{\bar{\otimes}}\cdots\mathbin{\bar{\otimes}} \tilde{N}_n$. 
For a $\mathrm{Tr}$-finite projection $p\in L\mathbb{R}$, we have $pC_\phi(\tilde{N})p= p\tilde{p}C_\phi(N)\tilde{p}p\subset C_h(M)$. 
We apply Proposition \ref{location core} to $pC_\phi(\tilde{N}_2\mathbin{\bar{\otimes}} \cdots\mathbin{\bar{\otimes}} \tilde{N}_{n})p\subset C_h(M)$ 
and get $pC_\phi(\tilde{N}_2\mathbin{\bar{\otimes}} \cdots\mathbin{\bar{\otimes}} \tilde{N}_{n})p\preceq_{C_h(M)} C_h(M_{X_1})$, where $X_1=\{1,\ldots,m\}\setminus i$ for some $i$, by non-amenability of $(\tilde{N}_1)_{\phi_1}$. For simplicity we assume $i=1$. 
Note that $(\tilde{N}_n)_\phi=\tilde{p}_n(N_n)_\phi \tilde{p}_n$ has no amenable summand and hence is of type II. 
By Corollary \ref{Popa embed4}, there exist nonzero projections $e_1\in p(\tilde{N}_n)_\phi p$, $f_1\in C_h(M_{X_1})$, and a partial isometry $V_1\in \langle C_h(M),C_h(M_{X_1}) \rangle$ 
such that $\mathrm{Tr}(f_1)<\infty$, $V_1^*V_1\in e_1C_\phi(\tilde{N}_2\mathbin{\bar{\otimes}} \cdots\mathbin{\bar{\otimes}} \tilde{N}_{n})'$, and $V_1V_1^*= f_1e_{X_1}$, where $e_{X_1}$ is the Jones projection for $C_h(M_{X_1})\subset C_h(M)$. 
We have 
\begin{eqnarray*}
V_1 e_1C_\phi(\tilde{N}_2\mathbin{\bar{\otimes}} \cdots\mathbin{\bar{\otimes}} \tilde{N}_{n})e_1 V_1^*
&\subset& f_1e_{X_1}\langle C_h(M),C_h(M_{X_1}) \rangle f_1e_{X_1}\\
&=& \mathbb{C}e_{M_1}\mathbin{\bar{\otimes}} f_1C_h(M_{X_1})f_1, 
\end{eqnarray*}
where $e_{M_1}$ is the projection from $L^2(M_1)$ onto $\mathbb{C}\hat{1}_{M_1}$. 
We again apply Proposition \ref{location core} to the inclusion and get $V_1 e_1C_\phi(\tilde{N}_3\mathbin{\bar{\otimes}} \cdots\mathbin{\bar{\otimes}} \tilde{N}_{n})e_1 V_1^*\preceq_{\mathbb{C}e_{M_1}\mathbin{\bar{\otimes}} C_h(M_{X_1})}\mathbb{C}e_{M_1}\mathbin{\bar{\otimes}} C_h(M_{X_2})$ for some $X_2$ (since $(\tilde{N}_2)_\phi$ has no amenable summand). We assume $X_2=X_1\setminus \{2\}$. 
Then there exist nonzero projections $e_2\in e_1(\tilde{N}_{n})_\phi e_1$, $f_2\in  C_h(M_{X_2})$, and a partial isometry $V_2\in \mathbb{C}e_{M_1}\mathbin{\bar{\otimes}} \langle C_h(M_{X_1}),C_h(M_{X_2}) \rangle$ 
such that $\mathrm{Tr}(f_2)<\infty$, $V_2^*V_2\in V_1e_2V_1^* (V_1C_\phi(\tilde{N}_3\mathbin{\bar{\otimes}} \cdots\mathbin{\bar{\otimes}} \tilde{N}_{n})V_1^*)'$, $V_2V_2^*= \mathbb{C}e_{M_1}\mathbin{\bar{\otimes}} f_2e_{X_2}^{X_1}$, where $e_{X_2}^{X_1}$ is the Jones projection for $C_h(M_{X_2})\subset C_h(M_{X_1})$. 
We get 
\begin{eqnarray*}
V_2V_1 e_2C_\phi(\tilde{N}_3\mathbin{\bar{\otimes}} \cdots\mathbin{\bar{\otimes}} \tilde{N}_{n})e_2 V_1^*V_2^*
\subset \mathbb{C}e_{M_1}\mathbin{\bar{\otimes}} \mathbb{C}e_{M_2} \mathbin{\bar{\otimes}} f_2C_h(M_{X_2})f_2, 
\end{eqnarray*}
where $e_{M_2}$ is the projection from $L^2(M_2)$ onto $\mathbb{C}\hat{1}_{M_2}$. 
Since $V_2V_1\in \langle C_h(M),C_h(M_{X_2}) \rangle$, $(V_2V_1)^*V_2V_1\in e_2 C_\phi(\tilde{N}_3\mathbin{\bar{\otimes}} \cdots\mathbin{\bar{\otimes}} \tilde{N}_{n})'$, and $V_2V_1(V_2V_1)^*=V_2V_2^*=e_{X_2}f_2$, where $e_{X_2}$ is the Jones projection for $C_h(M)\subset C_h(M_{X_1})$, we have $pC_\phi(\tilde{N}_3\mathbin{\bar{\otimes}} \cdots\mathbin{\bar{\otimes}} \tilde{N}_{n})p\preceq_{C_h(M)} C_h(M_{X_2})$. 
Hence $qC_\phi({N}_3\mathbin{\bar{\otimes}} \cdots\mathbin{\bar{\otimes}} {N}_{n})q\preceq_{C_h(M)} C_h(M_{X_2})$. 
Repeating this operation, we can prove the lemma.
\end{proof}

\begin{Rem}\upshape\label{reduction rem}
In the proof of the lemma, we in fact proved the following statement: 
if $qC_\phi(N_Y)q\preceq_{C(M)} C_h(M_X)$ (resp.\ $q(N_Y)_\phi q\preceq_{C(M)} C_h(M_X)$) for some $X,Y$ and $q=p_1\otimes \cdots \otimes p_n \otimes p$, where $p\in L\mathbb{R}_N$ is a Tr-finite projection and $p_i\in (N_i)_\phi$ are $\phi$-finite projections with $p_i(N_i)_\phi p_i$ non-amenable, then for any $i\in Y$ there is some $j\in X$ such that $qC_\phi(N_{Y\setminus\{i\}})q\preceq_{C(M)} C_h(M_{X\setminus\{j\}})$ (resp.\ $q(N_{Y\setminus\{i\}})_\phi q\preceq_{C(M)} C_h(M_{X\setminus\{j\}})$).
\end{Rem}

\begin{Cor}\label{first half A}
If $q(N_Y)_\phi q\preceq_{C(M)} C_h(M_X)$ for some $X,Y$, where $q$ is as in the previous remark, then $|Y|\leq|X|$.
\end{Cor}
\begin{proof}
Suppose $|Y|>|X|$. By the remark above, we get $q(N_{\tilde{Y}})_\phi q\preceq_{C_h(M)} L\mathbb{R}$ for some $\tilde{Y}\neq\emptyset$. This is a contradiction since $q(N_{\tilde{Y}})_\phi q$ has no amenable summand. 
\end{proof}

It is now easy to prove the first half of Theorem \ref{A}, since $qC_\phi(N_Y)q \preceq_{C(M)} C_h(M_X)$ implies $q(N_Y)_\phi q\preceq_{C(M)} C_h(M_X)$.

We next assume $M_i=L^\infty(\mathbb{G}_i)$ and $N_j$ satisfy conditions in the second statement in Theorem \ref{A}. 
In the case, $M$ is full from Lemmas \ref{fullness of tensor factors} and \ref{C is full}, and hence each $N_i$ is also full. 
So $N_i=(N_i)_{\phi_i}$ is a non-amenable $\rm II_1$ factor if $N_i$ is a $\rm II_1$ factor. 
If $N_i$ is a $\rm III_1$ factor, as we mentioned in Subsection \ref{discrete decomposition}, $\phi_i$ is $\mathrm{Sd}(N_i)$-almost periodic and the discrete core $D_{\phi_i}(N_i)$ is a $\rm II_\infty$ factor and is isomorphic to $(N_i)_{\phi_i}\mathbin{\bar{\otimes}} \mathbb{B}(H)$ for some $H$. Hence $(N_i)_{\phi_i}$ is a non-amenable $\rm II_1$ factor. 
Thus we can apply Lemma \ref{reduction lemma} and get that for any Tr-finite projection $p\in L\mathbb{R}_N$ and any $i$ there is some $j$ such that $pC_\phi(N_i)p\preceq_{C_h(M)}C_h(M_j)$. 
We first prove this correspondence is one to one.
For simplicity, from now on we assume that there is at least one $\rm III_1$ factor among $M_i$. 

\begin{Lem}
Let $X,Y$ be subsets in $\{1,\ldots,n\}$ and $p\in L\mathbb{R}_N$ a $\mathrm{Tr}$-finite projection.  
\begin{itemize}
	\item[$\rm (i)$] If $N_Y$ is a $\rm II_1$ factor, then $(N_Y)_\phi p\preceq_{C_h(M)}C_h(M_X)$ implies $(N_Y)_\phi p\subset_{\rm approx}C_h(M_X)$. 
	\item[$\rm (ii)$] If $N_Y$ is a $\rm III_1$ factor, then $pC_\phi(N_Y)p\preceq_{C_h(M)}C_h(M_X)$ implies $pC_\phi(N_Y)p\subset_{\rm approx}C_h(M_X)$. 
\end{itemize}
\end{Lem}
\begin{proof}
Use Lemmas \ref{full tensor}, \ref{2.9}, and \ref{relative commutant in cores}.
\end{proof}

\begin{Lem}\label{uniqueness1}
Let $X_1,X_2$, and $Y$ be subsets in $\{1,\ldots,n\}$ with $|X_1|=|X_2|=|Y|$ and let $p_1,p_2\in L\mathbb{R}_N$ be $\mathrm{Tr}$-finite projections. 
Assume that $p_i(N_Y)_\phi p_i\preceq_{C_h(M)}C_h(M_{X_i})$ for $i=1,2$ when $N_Y$ is a $\rm II_1$ factor, or $p_iC_\phi(N_{Y})p_i\preceq_{C_h(M)}C_h(M_{X_i})$ for $i=1,2$ when $N_Y$ is a $\rm III_1$ factor. 
Then $X_1=X_2$. 
\end{Lem}
\begin{proof}
Replacing $p_1$ and $p_2$ with $p_1\vee p_2$, we may assume $p_1=p_2(=:p)$. 
By the previous lemma, we have $(N_Y)_\phi p \subset_{\rm approx}C_h(M_{X_a})$ for $a=1,2$. 
Lemma \ref{2.7} (and Remark \ref{2.8}) then implies $(N_Y)_\phi p\subset_{\rm approx} C_h(M_{X_1\cap X_2})$. Thus we get $(N_Y)_\phi p\preceq_{C_h(M)}C_h(M_{X_1\cap X_2})$. 
This contradicts to Corollary \ref{first half A} when $X_1\neq X_2$, since $|X_1\cap X_2|<|Y|$.
\end{proof}

\begin{Lem}\label{uniqueness2}
Let $X,Y_1$, and $Y_2$ be subsets in $\{1,\ldots,n\}$ with $|X|=|Y_1|=|Y_2|$ and let $p_1,p_2\in L\mathbb{R}_N$ be $\mathrm{Tr}$-finite projections. 
If $p_1C_\phi(N_{Y_1})p_1\preceq_{C_h(M)}C_h(M_X)$ and $p_2C_\phi(N_{Y_2})p_2\preceq_{C_h(M)}C_h(M_X)$, then $Y_1=Y_2$. 
\end{Lem}
\begin{proof}
Replacing $p_1$ and $p_2$ with $p_1\vee p_2$, we may assume $p_1=p_2(=:p)$. 
By Lemma \ref{reduction lemma}, there is some $X_2$ with $|X_2|=|Y_1^c|$ such that $pC_\phi(N_{Y_1^c}) p\preceq_{C_h(M)}C_h(M_{X_2})$. 
We claim that $X\cap X_2\neq\emptyset$ when $Y_1\neq Y_2$. 
Take $i\in Y_1^c\cap Y_2$. Since $i\in Y_1^c$, by Remark \ref{relative commutant in cores}, we have $pC_\phi(N_i) p\preceq_{C_h(M)}C_h(M_j)$ for some $j\in X_2$. 
Since $i\in Y_2$, we also have $pC_\phi(N_i) p\preceq_{C_h(M)}C_h(M_l)$ for some $l\in X$ by Remark \ref{relative commutant in cores} and assumption. By the previous lemma, $j= l\in X\cap X_2$.

By assumption, there is $q\in L\mathbb{R}_M$ such that $pC_\phi(N_{Y_1})p\preceq_{rC_h(M)r}qC_h(M_X)q$ for $r:=p\vee q$. 
By \cite[\textrm{Lemma 3.5}]{Va08} and Lemma \ref{relative commutant in cores}, we have 
\begin{equation*}
qC_h(M_{X})'q\cap qC_h(M)q\preceq_{rC_h(M)r}pC_\phi(N_{Y_1})' p\cap pC_h(M)p
\end{equation*}
and 
\begin{eqnarray*}
(M_{X^c})_hq\subset qC_h(M_{X})'q\preceq pC_\phi(N_{Y_1})' p \subset pC_\phi(N_{Y_1^c})p\preceq qC_h(M_{X_2})q.
\end{eqnarray*}
When $N_{Y_1^c}$ is a $\rm III_1$ factor, the final embedding becomes $\subset_{\rm approx}$. 
When $N_{Y_1^c}$ is a $\rm II_1$ factor, since $N_{Y_1}$ is a $\rm III_1$ factor, we have $C_\phi(N_{Y_1})'=(N_{Y_1^c})_\phi$ and hence replacing the final embedding with $(N_{Y_1^c})_\phi p\preceq qC_h(M_{X_2})q$, we again get $\subset_{\rm approx}$. 
In any case, by \cite[\textrm{Lemma 3.8}]{Va08}, we get $(M_{X^c})_hq\preceq_{C_h(M)} C_h(M_{X_2})$. By Lemma \ref{2.5}, there is some $s\in (M_{X^c})_h'\cap qC_h(M)q$ such that $(M_{X^c})_hqs\subset_{\rm approx} C_h(M_{X_2})$. 
Since we trivially have $(M_{X^c})_hq\subset_{\rm approx} C_h(M_{X^c})$, it also holds that $(M_{X^c})_hqs\subset_{\rm approx} C_h(M_{X^c})$ by Lemma \ref{2.3}. 
Finally from Remark \ref{2.9}, we get $(M_{X^c})_hqs\subset_{\rm approx} C_h(M_{X^c\cap X_2})$. Since $|X^c\cap X_2|<|X^c|$ by the claim above, we get a contradiction from Corollary \ref{first half A}.
\end{proof}

Thanks for previous two lemmas, there is a unique $\sigma\in  \mathfrak{S}_n$ such that
\begin{equation*}
pC_\phi(N_i)p \preceq_{C_h(M)}  C_h(M_{\sigma(i)}) \quad (i=1,\ldots,n)
\end{equation*}
for any Tr-finite projection $p\in L\mathbb{R}_N$. 
In this case, we in fact have
\begin{equation*}
pC_\phi(N_Y)p \preceq_{C_h(M)}  C_h(M_{\sigma(Y)}) \quad (Y\subset \{1,\ldots,n\}).
\end{equation*}
To see this, fix any $Y\subset \{1,\ldots,n\}$. We use Lemma \ref{reduction lemma} and find some $X\subset \{1,\ldots,n\}$, which is unique by Lemma \ref{uniqueness1}, such that $pC_\phi(N_Y)p \preceq_{C_h(M)}  C_h(M_X)$ and $|X|=|Y|$. By Remark \ref{reduction rem}, for any $i\in Y$ there is $j\in X$ such that $pC_\phi(N_i)p \preceq_{C_h(M)}  C_h(M_j)$. This implies $\sigma(i)=j$ and, since $\sigma$ is bijective and $|X|=|Y|$, we get $\sigma(Y)=X$.

We next take the relative commutants of the embedding (with respect to $Y^c$) and get 
\begin{equation*}
(M_{\sigma(Y)})_h q\subset qC_h(M_{\sigma(Y^c)})'q\preceq_{C_h(M)} C_\phi(N_{Y^c})'\subset (N_Y)_\phi\mathbin{\bar{\otimes}} L\mathbb{R}
\end{equation*}
for some Tr-finite projection $q\in L\mathbb{R}_M$. 
Then by (the proof of) \cite[\textrm{Proposition 2.10}]{BHR12}, we finally get 
\begin{equation*}
(M_{\sigma(Y)})_h\preceq_{M} (N_Y)_\phi  \quad (Y\subset \{1,\ldots,n\}).
\end{equation*}

\begin{Lem}\label{II_1 correspondence}
In the setting, $M_{\sigma(i)}$ is a $\rm II_1$ factor if and only if so is $N_i$.
\end{Lem}
\begin{proof}
Suppose that $M_{\sigma(i)}$ is a $\rm II_1$ factor. Then since 
\begin{equation*}
(N_i)_\phi \mathbin{\bar{\otimes}} L\mathbb{R}p\subset pC_\phi(N_i)p \preceq_{C_h(M)}  C_h(M_{\sigma(i)})=(M_{\sigma(i)})_h\mathbin{\bar{\otimes}} L\mathbb{R}
\end{equation*} 
for any Tr-finite projection $p\in L\mathbb{R}_N$, 
we have $(N_i)_\phi \preceq_{M} (M_{\sigma(i)})_h$ by \cite[\textrm{Proposition 2.10}]{BHR12}. 
If $N_i$ is of type III, we can apply Lemma \ref{relative commutant} two times and get 
\begin{equation*}
N_i=((N_i)_\phi'\cap M)'\cap M \preceq_{M} ((M_{\sigma(i)})_h'\cap M)'\cap M=M_{\sigma(i)}.
\end{equation*}
Hence $N_i$ must contains a finite direct summand and hence a contradiction. 
Since $(M_{\sigma(i)})_h\preceq_{M} (N_i)_\phi$, the converse holds by the same argument. 
\end{proof}

\begin{proof}[\bf Proof of Theorem \ref{A}]
The first half statement was already proved by Corollary \ref{first half A}. So we see the second half. 

\begin{itemize}
	\item {\bf  Case 1: all $\bf M_i$ are $\bf \rm II_1$ factors.} 
\end{itemize}

In the case, we can apply the prime factorization result of Ozawa and Popa. 
So there are $u\in \mathcal{U}(M)$, $\sigma\in \mathfrak{S}_n$, and $t_i>0$ such that $uN_iu^*=M_{\sigma(i)}^{t_i}$ (which implies $N_i\preceq_M M_{\sigma(i)}$). 
By the same method as in Lemmas \ref{uniqueness1} and \ref{uniqueness2}, we can prove that $\sigma\in \mathfrak{S}_n$ is unique with the condition $N_i\preceq_M M_{\sigma(i)}$ for all $i$. 
On the other hand, by Lemma \ref{reduction lemma}, for any $i$ there is some $j$ such that $N_i\mathbin{\bar{\otimes}} L\mathbb{R}p=pC_\phi(N_i)p \preceq_{C_h(M)}  C_h(M_j)= M_j\mathbin{\bar{\otimes}} L\mathbb{R}$ for any Tr-finite projection $p\in L\mathbb{R}_N$. 
By \cite[\textrm{Proposition 2.10}]{BHR12}, this implies $N_i\preceq_{M} M_j$. 
So by the uniqueness, we get $\sigma(i)=j$ and $\sigma$ is determined from $pC_\phi(N_i)p \preceq_{C_h(M)}  C_h(M_{\sigma(i)})$. 

\bigskip

We next consider the case that there is at least one $\rm III_1$ factor among $M_i$. 
As we mentioned before Lemma \ref{II_1 correspondence}, there is a unique $\sigma\in  \mathfrak{S}_n$ with $pC_\phi(N_i)p \preceq_{C_h(M)}  C_h(M_{\sigma(i)})$, 
and it then holds that $(M_{\sigma(i)})_h\preceq_{M} (N_i)_\phi$.

\begin{itemize}
	\item {\bf  Case 2: there are more than two $\bf \rm III_1$ factors among $\bf M_i$.} 
\end{itemize}

By Lemma \ref{II_1 correspondence}, the same is true for $N_i$. For simplicity we assume $\sigma=\mathrm{id}$. 
When $M_{i}$ and $N_i$ are $\rm III_1$ factors, since so are $M_{\{i\}^c}$ and $N_{\{i\}^c}$, we can apply Lemma \ref{relative commutant} to $(M_{i})_h\preceq_{M} (N_i)_\phi$ two times and get $M_{i}\preceq_{M} N_i$. 
Also $M_{i}\preceq_{M} N_i$ holds for $\rm II_1$ factors $M_i$ and $N_i$. 
We fix $i$ and apply Lemma \ref{hariawase2} and get a partial isometry or a unitary $u\in M$ such that $uM_{i}u^*\subset uu^*N_iuu^*$. Write $p=uu^*\in N_i$ and $q=u^*u\in M_i$. 
We have 
\begin{equation*}
uM_iu^*\mathbin{\bar{\otimes}} uM_{\{i\}^c}u^*= uMu^*=pMp=pN_ip\mathbin{\bar{\otimes}} N_{\{i\}^c},
\end{equation*}
and hence
\begin{equation*}
uM_{\{i\}^c}u^*=(uM_iu^*)'\cap pMp= ((uM_iu^*)'\cap pN_ip)\mathbin{\bar{\otimes}} N_{\{i\}^c}.
\end{equation*}
Since $M_i\simeq uM_iu^*$, we can apply the first half of Theorem \ref{A}, and then get amenability of $((uM_iu^*)'\cap pN_ip)$. 
Since $M_{\{i\}^c}$ is full from Lemmas \ref{fullness of tensor factors} and \ref{C is full}, $((uM_iu^*)'\cap pN_ip)$ must be a type I factor, say $\mathbb{B}(K)$. 
Since $uMu^*=uM_iu^*\mathbin{\bar{\otimes}} uM_{\{i\}^c}u^*= uM_iu^*\mathbin{\bar{\otimes}} \mathbb{B}(K)\mathbin{\bar{\otimes}} N_{\{i\}^c}$, we get $pN_ip=(N_{\{i\}^c})'\cap pMp=uM_iu^*\mathbin{\bar{\otimes}} \mathbb{B}(K)$. 
This means $N_i$ and $M_i$ are isomorphic when they are $\rm III_1$ factors, and are stably isomorphic when they are $\rm II_1$ factors. 

\begin{itemize}
	\item {\bf  Case 3: there is only one $\bf\rm III_1$ factor among $\bf M_i$.} 
\end{itemize}

The same is true for $N_i$. Assume for simplicity that $\sigma=\mathrm{id}$ and $M_1$ and $N_1$ are $\rm III_1$ factors. 
For $i\neq 1$, since $M_i\preceq_M N_i$, $M_i$ and $N_i$ are stably isomorphic by the same manner as above. 
So we see the case $i=1$. Put $X:=\{1\}^c$. 
Since $N_X$ and $M_X$ are also $\rm II_1$ factors, we have  $M_X=(M_X)_h\preceq_{M} (N_X)_\phi=N_X$. 
So there is a partial isometry $u\in M$ such that $uM_Xu^*\subset pN_Xp$, where $p:=uu^*$. 
We have $uM_1u^*=((uM_Xu^*)'\cap pN_Xp)\mathbin{\bar{\otimes}} N_1$. The algebra $(uM_Xu^*)'\cap pN_Xp$ must be a type I factor and hence $M_1$ and $N_1$ are isomorphic. 
\end{proof}

%%%%%%%%%%%%%%%%%%%%%%%%%%%%%%%%%%%%
\section{\bf Prime factorization results for crossed product algebras}
%%%%%%%%%%%%%%%%%%%%%%%%%%%%%%%%%%%%

In the proof of Theorem \ref{B}, we go along a similar line to the proof of Theorem \ref{A}. 
We first prove a key proposition with condition (AO) and then study intertwiners on crossed products. 
We finally prove some one-to-one correspondence between each component of direct product groups. 

%%%%%%%%%%%%%%%%%%%%%%%%%%%%%%%%%%%%
\subsection{\bf Location of subalgebras}
%%%%%%%%%%%%%%%%%%%%%%%%%%%%%%%%%%%%

Let $\Gamma$ and $\Lambda$ be discrete groups acting on semifinite tracial von Neumann algebras $(A,\mathrm{Tr}_A)$ and $(B,\tau_B)$ with $\tau_B(1)=1$ as trace preserving actions. 
Assume that we have an inclusion $B\rtimes \Lambda\subset p(A\rtimes \Gamma)p$  with $B\subset pAp$ (not necessarily $B=pAp$) for some $\mathrm{Tr}_A$-finite projection $p\in \mathcal{Z}(A)$. 
Write $M:=A\rtimes\Gamma$.

\begin{Lem}\label{bi-exact lemma}
Let $\Lambda_0\subset \Lambda$ and $\Gamma_0\subset \Gamma$ be subgroups and $q\in \mathcal{Z}(A)$ a $\mathrm{Tr}_A$-finite projection with $p\leq q$. 
If $L\Lambda_0\preceq_{qMq} q(A\rtimes \Gamma_0)q$, then we have $B\rtimes\Lambda_0\preceq_{qMq}q(A\rtimes \Gamma_0)q$.
\end{Lem}
\begin{proof}
By assumption, there exist $\delta>0$ and a finite subset $\mathcal{F}\subset pMq$ such that $\sum_{x,y\in\mathcal{F}}\|E_{A\rtimes \Gamma_0}(y^*wx)\|^2_2> \delta$ for all $w\in\mathcal{U}(L\Lambda_0)$. 
We may assume that $\mathcal{F}$ consists of elements of the form $p\lambda_s q$ for some $s\in\Gamma$. 
Put $d_0:=\sum_{x\in\mathcal{F}}xe_{A\rtimes \Gamma_0}x^*$. By the proof of (i) $\Rightarrow$ (ii) in Proposition \ref{Popa embed2}, $d_0$ satisfies condition (ii) in the proposition. 
Since $d_0$ commutes with $B\subset pAp$, $d_0$ actually satisfies that the $\sigma$-weak closure of 
\begin{equation*}
\mathrm{co}\{ wd_0w^*\mid w=\lambda_s b \textrm{ for some } s\in\Lambda_0, b\in \mathcal{U}(B)\}
\end{equation*}
does not contain zero.  
Then by the same manner as in the proof of (ii) $\Rightarrow$ (iii) in Proposition \ref{Popa embed2}, we get $d\in \langle M, A\rtimes \Gamma_0\rangle\cap (L\Lambda_0\cup B)'$ satisfying condition (iii). 
Since $(L\Lambda_0\cup B)'=(B\rtimes \Lambda_0)'$, we get  $B\rtimes\Lambda_0\preceq_{qMq}q(A\rtimes \Gamma_0)q$.
\end{proof}

\begin{Pro}\label{bi-exact theorem}
Assume that $\mathrm{Tr}_A|_{\mathcal{Z}(A)}$ is semifinite and $\Gamma$ is bi-exact relative to $\mathcal{G}$, where $\mathcal{G}$ is a countable family of subgroups of $\Gamma$. Assume either $A$ is amenable or $\Gamma$ is weakly amenable. 
Then for any subgroup $\Lambda_0\subset \Lambda$, we have either one of the following conditions: 
\begin{itemize}
	\item[$(\rm i)$] There is a conditional expectation from $p\langle M,A\rangle p$ onto $L\Lambda_0'\cap pMp$, which restricts to the trace preserving expectation on $pMp$.
	\item[$(\rm ii)$] There exists $\Gamma_0 \in \mathcal{G}$ and a $\mathrm{Tr}_A$-finite projection $q\in \mathcal{Z}(A)$ with $p\leq q$ 
such that  $B\rtimes\Lambda_0\preceq_{qMq}q(A\rtimes\Gamma_0)q$.
\end{itemize}
\end{Pro}
\begin{proof}
Since $\mathrm{Tr}_A|_{\mathcal{Z}(A)}$ is semifinite, condition (ii) is equivalent to that there exists $\Gamma_0 \in \mathcal{G}$ such that $B\rtimes\Lambda_0\preceq_{M}A\rtimes\Gamma_0$. 

Suppose that $A$ is amenable. Then since $p\langle M,A\rangle p$ is also amenable, condition (i) means amenability of $L\Lambda_0'\cap pMp$. Assume that $L\Lambda_0'\cap pMp$ is non-amenable. 
Then by \cite[\textrm{Theorem 5.3.3 and Remark 5.3.4}]{Is12_1}, there exists $\Gamma_0 \in \mathcal{G}$ and a projection $q\in \mathcal{Z}(A)$ with $\mathrm{Tr}_A(q)<\infty$ 
such that $L\Lambda_0\preceq_{rMr}q(A\rtimes\Gamma_0)q$, where $r:=p\vee q$. Exchanging $r$ with $q$, we may assume $p\leq q$ and $r=q$. 
By the last lemma, we get $B\rtimes\Lambda_0\preceq_{qMq}q(A\rtimes\Gamma_0)q$.

Suppose next $\Gamma$ is weakly amenable and $B\rtimes\Lambda_0\not\preceq_{M}A\rtimes\Gamma_0$ for any $\Gamma_0 \in \mathcal{G}$. 
This means $L\Lambda_0\not\preceq_{M}A\rtimes\Gamma_0$ for any $\Gamma_0 \in \mathcal{G}$ by the previous lemma. 
Then by \cite[\textrm{Proposition 5.2.4}]{Is12_1}, there is a unital diffuse abelian subalgebra $A_0\subset L\Lambda_0$ such that $A_0\not\preceq_{M}A\rtimes\Gamma_0$ for any $\Gamma_0 \in \mathcal{G}$. 
By \cite[\textrm{Theorem 3.1}]{PV12}, there is a conditional expectation from $p\langle M,A\rangle p$ onto $\mathcal{N}_{pMp}(A_0)''$, which restricts to the trace preserving expectation on $pMp$. 
Since $\mathcal{N}_{pMp}(A_0)''$ contains $L\Lambda_0'\cap pMp$, 
we are done.
\end{proof}

%%%%%%%%%%%%%%%%%%%%%%%%%%%%%%%%%%%%
\subsection{\bf Intertwiners inside crossed product von Neumann algebras}\label{Intertwiners inside crossed product von Neumann algebras}
%%%%%%%%%%%%%%%%%%%%%%%%%%%%%%%%%%%%

In the subsection, we study intertwiners in crossed product von Neumann algebras. 
Under the assumption of freeness of actions, we can construct a special form of an embedding and an intertwiner. 

\begin{Lem}\label{proj}
Let $A$ be a commutative von Neumann algebra and $p_i\in A$ $(i\in\mathbb{N})$ be projections. If $\sum_ip_i\leq n$ for some $n\in\mathbb{N}$, then there exists a projection $z\in A$ such that $\sum_izp_i\neq 0$ and $zp_i=0$ except for finitely many $i$.
\end{Lem}
\begin{proof}
If $p_1$ is orthogonal to any other $p_i$, then $p_1$ does the work. If not, then we find the minimum number $i_2$ in the set $\{j\mid p_1p_j\neq0, j\neq 1\}$. If $p_1p_{i_2}$ is orthogonal to any other $p_i$, then $p_1p_{i_2}$ does the work. Repeating this argument, we get $p_1p_{i_2}\cdots p_{i_m}$ ($m\leq n$) which is orthogonal to any other $p_i$. 
In fact, if we have a non-zero projection $p_1p_{i_2}\cdots p_{i_m}$ for $n+1\leq m$, then we have $m\cdot p_1p_{i_2}\cdots p_{i_m}\leq p_1+p_{i_2}+\cdots +p_{i_m}\leq n$. Hence a contradiction. 
\end{proof}

\begin{Lem}\label{free}
Let $\Gamma$ be a discrete group and $\alpha$ a free action on a commutative von Neumann algebra $A$. For any non-zero projection $p\in A$  and any distinct elements $s_1,\ldots,s_n\in\Gamma$, there is a non-zero projection $q\in A$ such that $q\leq p$ and $\alpha_{s_i}(q)\alpha_{s_j}(q)=0$ for any $i,j$ with $i\neq j$. \end{Lem}
\begin{proof}Write $A=L^\infty(X,\mu)$ for some standard probability space $(X,\mu)$. We may assume that $\Gamma$ acts on $(X,\mu)$ as a free action. Put $X_g:=\{x\in X\mid g\cdot x=x\}$ for any $g\in \Gamma$. Then by freeness, we have $\mu(X_{s_i^{-1}s_j})=0$ for $i,j$ with $i\neq j$. Hence we have $\mu(\cup_{i\neq j}X_{s_i^{-1}s_j})=0$. This implies $Y\cap(\cup_{i\neq j}X_{s_i^{-1}s_j})^c$ is non-null for any non-null set $Y\subset X$. 

Since $X$ is a standard Borel space, it is isomorphic to $\mathbb{R}$ (or a countable set). Let $E_k$ $(k\in\mathbb{N})$ be a countable basis of open subsets of $\mathbb{R}$ (or the countable set). 
Let $E_k\subset X$ $(k\in\mathbb{N})$ be Borel subsets satisfying for any $x_1,\ldots,x_{n^2}\in X$ with $x_1\neq x_i$ for any $i\neq 1$, there exists $E_k$ such that $x_1\in E_k$ and $x_i\not\in E_k$ for any $i\neq1$ (e.g.\ take a basis of open subsets of $\mathbb{R}$). 
Then we have 
\begin{eqnarray*}
(\cup_{i\neq j}X_{s_i^{-1}s_j})^c
&=&\cap_{i\neq j}X_{s_i^{-1}s_j}^c\\
&=& \{x\in X\mid s_i^{-1}s_j\cdot x\neq x \textrm{ for any $i\neq j$}\}\\
&=& \cup_k (E_k\setminus (\cup_{i\neq j}s_j^{-1}s_iE_k)).
\end{eqnarray*}
Let $Y\subset X$ be a non-null set whose characteristic function is $p$. Then we can find some $k$ such that $Y\cap(E_k\setminus (\cup_{i\neq j}s_j^{-1}s_iE_k))$ is non-null. Let $q$ be the projection corresponding to this set. Then it satisfies $q\leq p$ and $\alpha_{s_j^{-1}s_i}(q)q=0$ for $i\neq j$.
\end{proof}

\begin{Lem}\label{projections in Cartan}
Let $M$ be a $\rm II_1$ factor and $A\subset M$ be a regular von Neumann subalgebra (namely, $\mathcal{N}_M(A)''=M$). If non-zero projections $e,f\in\mathcal{Z}(A)$ have the same trace in $M$, then there exists $u\in\mathcal{N}_M(A)$ such that $ueu^*=f$.
\end{Lem}
\begin{proof}
See the proof of \cite[\textrm{Lemma F.16}]{BO08}.
\end{proof}

\begin{Lem}
Let $\Gamma$ be a discrete group and $\alpha$ an action on a von Neumann algebra $A$. 
Let $\Gamma_0$ be a (possibly trivial) subgroup of $\Gamma$ and $\sigma\colon \Gamma/\Gamma_0 \rightarrow\Gamma$ a section with $\sigma(\Gamma_0)=e$, where $e$ is the unit of $\Gamma$. 
Consider the unitary element 
\begin{equation*}
U_\sigma\colon L^2(A)\otimes\ell^2(\Gamma)\rightarrow L^2(A)\otimes \ell^2(\Gamma_0)\otimes\ell^2(\Gamma/\Gamma_0); \ \hat{a}\otimes \delta_{\sigma(s\Gamma_0)g}\mapsto \hat{a}\otimes \delta_{g}\otimes \delta_{s\Gamma_0}.
\end{equation*}
Then $U_\sigma$ gives an identification
\begin{eqnarray*}
\mathrm{Ad}U_\sigma\colon  \langle A\rtimes\Gamma, A\rtimes\Gamma_0\rangle\xrightarrow{\sim} (A\rtimes \Gamma_0) \mathbin{\bar{\otimes}} \mathbb{B}(\ell^2(\Gamma/\Gamma_0)).
\end{eqnarray*}
We have $\mathrm{Ad}U_\sigma(a)=\sum_{s\Gamma_0\in \Gamma/\Gamma_0}\alpha_{\sigma(s\Gamma_0)}^{-1}(a)\otimes e_{s\Gamma_0}$ for $a\in A$, where $e_{s\Gamma_0}$ is the orthogonal projection onto $\mathbb{C}\delta_{s\Gamma_0}$. 
If the action $\alpha$ is free on $A$, we have 
\begin{equation*}
\mathrm{Ad}U_\sigma\colon A'\cap \langle A\rtimes \Gamma,A\rtimes\Gamma_0\rangle \xrightarrow{\sim} \mathcal{Z}(A)\mathbin{\bar{\otimes}}\ell^\infty(\Gamma/\Gamma_0).
\end{equation*}
\end{Lem}
\begin{proof}
Since the right action of $A\rtimes\Gamma_0$ satisfies  $U_{\sigma}x^{\rm op}U_{\sigma}^*=x^{\rm op}\mathbin{\bar{\otimes}} 1_{\mathbb{B}(\ell^2(\Gamma/\Gamma_0))}$ for $x\in A\rtimes\Gamma_0$, we have 
\begin{eqnarray*}
U_{\sigma}\langle A\rtimes\Gamma, A\rtimes\Gamma_0\rangle U_{\sigma}^*=U_{\sigma}((A\rtimes\Gamma_0)^{\rm op})'U_{\sigma}^*
&=&((A\rtimes\Gamma_0)^{\rm op})'\mathbin{\bar{\otimes}} \mathbb{B}(\ell^2(\Gamma/\Gamma_0))\\
&=&(A\rtimes\Gamma_0)\mathbin{\bar{\otimes}}\mathbb{B}(\ell^2(\Gamma/\Gamma_0)).
\end{eqnarray*}
The formula of $\mathrm{Ad}U_\sigma(a)$ for $a\in A$ follows from a direct calculation.

Next assume that the action is free. Let $x$ be in  $(A\rtimes\Gamma_0)\mathbin{\bar{\otimes}}\mathbb{B}(\ell^2(\Gamma/\Gamma_0))$. Write it as $x=\sum_{s,t\in \Gamma/\Gamma_0}x_{s,t}\otimes e_{s,t}$, where $x_{s,t}\in A\rtimes\Gamma_0$ and $e_{s,t}$ is the matrix unit along $s,t\in \Gamma/\Gamma_0$. 
By the first part of the statement, elements of $A$ in $(A\rtimes\Gamma_0)\mathbin{\bar{\otimes}}\mathbb{B}(\ell^2(\Gamma/\Gamma_0))$ is of the form $a=\sum_{s\in\Gamma/\Gamma_0}a_{s,s}\otimes e_{s,s}$, where $a_{s,s}= \alpha_{\sigma(s)}^{-1}(a)$. 
If $x$ commutes $a\in A$, a simple calculation shows that $x_{s,t}a_{t,t}=a_{s,s}x_{s,t}$ for any $s,t\in\Gamma/\Gamma_0$. 

When $s=t$, then we get $x_{s,s}\alpha_{\sigma(s)}^{-1}(a)=\alpha_{\sigma(s)}^{-1}(a)x_{s,s}$ for any $a\in A$. Hence by the freeness, we get $x_{s,s}\in A'\cap (A\rtimes \Gamma_0)=\mathcal{Z}(A)$. 
When $s\neq t$, then we have $x_{s,t}\alpha_{\sigma(t)}^{-1}(a)=\alpha_{\sigma(s)}^{-1}(a)x_{s,t}$ for any $a\in A$. This means $x_{s,t}a=\alpha_{k}^{-1}(a)x_{s,t}$ for all $a\in A$, where $k=\sigma(t)^{-1}\sigma(s)\neq e$. 
Let $x_{s,t}=\sum_{b\in\Gamma_0}\lambda_b (x_{s,t})_b$ be the Fourier expansion of $x_{s,t}$ in $A\rtimes \Gamma_0$. 
Then the equality means $(x_{s,t})_ba=\alpha_{kb}^{-1}(a)(x_{s,t})_b$ for any $b\in \Gamma_0$ and $a\in A$. Since $kb$ is not the unit of $\Gamma$, we get $(x_{s,t})_b=0$ for $b\in \Gamma_0$ and hence $x_{s,t}=0$. 
Thus $x$ is contained in $\mathcal{Z}(A)\mathbin{\bar{\otimes}}\ell^\infty(\Gamma/\Gamma_0)$.
\end{proof}

From now on, we keep the following setting. 
Let $\Gamma$ be a discrete group and $\alpha$ a trace preserving free action of $\Gamma$ on a semifinite tracial von Neumann algebra $(A,\mathrm{Tr}_A)$ with $\mathrm{Tr}_A|_{\mathcal{Z}(A)}$ semifinite. Write $M:=A\rtimes \Gamma$. 
Let $\Gamma_0\subset \Gamma$ be a (possibly trivial) subgroup. 
Let $p\in \mathcal{Z}(A)$ be a projection with $\mathrm{Tr}_A(p)=1$ and $N\subset p(A\rtimes\Gamma)p$ a von Neumann subalgebra containing $pAp$. 
Assume either that $\mathcal{Z}(A)$ is diffuse, or $A$ is a $\rm II_1$ factor (so that $\mathrm{Tr}_A$ is finite and $p=1_A$). 
We fix $U_\sigma$ in the previous lemma for a section $\sigma$ and write $\tilde{s}:=\sigma(s)$ for $s\in\Gamma/\Gamma_0$. 

\begin{Pro}\label{intertwiner cross}
If $N\preceq_M A\rtimes \Gamma_0$, 
then there exist non-zero projections $e\in \mathcal{Z}(A)p$ and $d\in (N)'e \cap \langle M, A\rtimes\Gamma_0\rangle$ such that $U_\sigma dU_\sigma^*$ is contained in $(\mathbb{C}\mathbin{\bar{\otimes}} \ell^\infty(\mathcal{S}))U_\sigma eU_\sigma^*$ for a finite subset $\mathcal{S}\subset \Gamma/\Gamma_0$. 
In the case, choosing an appropriate $\mathcal{S}$, a $*$-homomorphism $\pi\colon eNe\ni x\mapsto U_\sigma dxU_\sigma^* \in (A\rtimes\Gamma_0)\mathbin{\bar{\otimes}}\mathbb{B}(\ell^2(\mathcal{S}))$ satisfies that $\pi(eAe)\subset A\mathbin{\bar{\otimes}} \ell^\infty(\mathcal{S})$ and $\pi(a)=\sum_{s\in\mathcal{S}} \alpha_{\tilde{s}}^{-1}(a)\otimes e_{s,s}$ for $a\in eAe$. 
\end{Pro}
\begin{proof}
By assumption there is an element $d\in N'\cap \langle M, A\rtimes\Gamma_0\rangle$ satisfying $d=dp$ and $\mathrm{Tr}_{\langle M,A\rtimes\Gamma_0\rangle}(d)<\infty$. 
Taking a spectral projection of $d$, we may assume $d$ is a projection. By the last lemma, $d$ is contained in 
\begin{equation*}
N'\cap p\langle M, A\rtimes\Gamma_0\rangle p 
\subset p(A'\cap \langle M, A\rtimes\Gamma_0\rangle)p
\simeq (\mathcal{Z}(A)\mathbin{\bar{\otimes}} \ell^\infty(\Gamma/\Gamma_0))U_\sigma pU_\sigma^*.
\end{equation*}
We write $U_\sigma dU_\sigma^*=(d_s)_{s\in\Gamma/\Gamma_0}$ as an element in $\mathcal{Z}(A)\mathbin{\bar{\otimes}} \ell^\infty(\Gamma/\Gamma_0)$. 
Note that the canonical trace on $\langle M, A\rtimes\Gamma_0\rangle$ is of the form $\mathrm{Tr}_{A\rtimes\Gamma_0}\otimes \mathrm{Tr}_{\ell^2(\Gamma/\Gamma_0)}$ on $\mathcal{Z}(A)\mathbin{\bar{\otimes}} \ell^\infty(\Gamma/\Gamma_0)$. 
When $A$ is a $\rm II_1$ factor and $p=1$, then since $d$ is a trace finite projection, it is contained in $\mathbb{C}\mathbin{\bar{\otimes}} \ell^\infty(\mathcal{S})$ for some finite $\mathcal{S}$. 
So we assume $\mathcal{Z}(A)$ is diffuse. 

We claim that there exists a projection $e\in \mathcal{Z}(A)p$ such that $U_\sigma deU_\sigma^*=(d_se_s)_s\neq0$ and $d_se_s=0$ except for finitely many $s\in\Gamma/\Gamma_0$, where $e_s=\alpha_{\tilde{s}}^{-1}(e)$. 
Let $\psi$ be an isomorphism from $\mathcal{Z}(A)\mathbin{\bar{\otimes}} \ell^\infty(\Gamma/\Gamma_0)$ onto itself given by $\psi((x_s)_s)=(\alpha_{\tilde{s}}(x_s))_s$. 
Put $(\tilde{d}_s)_s:=\psi((d_s)_s)=(\alpha_{\tilde{s}}(d_s))_s$. 
We regard $(\mathrm{id}_{\mathcal{Z}(A)}\otimes \mathrm{Tr}_{\ell^2(\Gamma/\Gamma_0)})((\tilde{d}_s)_s)=\sum_{s\in\Gamma/\Gamma_0}\tilde{d}_s (=:f)$ as a function on the spectrum of $\mathcal{Z}(A)$ which takes values on $[0,\infty]$. 
Note that $\mathrm{id}_{\mathcal{Z}(A)}\otimes \mathrm{Tr}_{\ell^2(\Gamma/\Gamma_0)}$ gives an extended center valued trace on $\mathcal{Z}(A)\mathbin{\bar{\otimes}} \mathbb{B}(\ell^2(\Gamma/\Gamma_0))$. 
Since $\sum_{s\in\Gamma/\Gamma_0}\mathrm{Tr}_A(\tilde{d}_s)=\sum_{s\in\Gamma/\Gamma_0}\mathrm{Tr}_A(d_s)=\mathrm{Tr}_{\langle M,A\rtimes\Gamma_0\rangle}(d)<\infty$, the function $f$ actually takes values on $[0, \infty)$. 
Let $z$ be a projection in $\mathcal{Z}(A)$ which corresponds to the characteristic function on $f^{-1}([0,n])$ for some large $n\in\mathbb{N}$. Then we have $0\neq\sum_{s\in \Gamma/\Gamma_0}\tilde{d}_sz=fz\leq n$ and hence the element $fz$ is contained in $\mathcal{Z}(A)$. 
Now by Lemma \ref{proj}, we can find a projection $w\in \mathcal{Z}(A)$ such that $fzw\neq0$ and $\tilde{d}_szw=0$ except for finitely many $s\in\Gamma/\Gamma_0$. 
Hence we get $0\neq \psi^{-1}((\tilde{d_s}zw)_s)=(d_s\alpha_{\sigma(s)}^{-1}(zw))_s$. 
Thus $e:=zw$ does the work. 

Finally replacing $e\in\mathcal{Z}(A)p$ in the claim with a sufficiently small one, we can assume $d_s\alpha_{\tilde{s}}^{-1}(e)=\alpha_{\tilde{s}}^{-1}(e)$ or $0$. 
Putting $\mathcal{S}:=\{s\mid d_s\alpha_{\tilde{s}}^{-1}(e)=\alpha_{\tilde{s}}^{-1}(e) \}$, we can end the proof. 
\end{proof}

For finite von Neumann algebras $A\subset M$, we say that the inclusion is a \textit{finite extension} if for any trace on $A$, the associated semifinite trace on $\langle M,A\rangle$ is finite. 
We use this notion in the next subsection with the following observation.  

Let $B\subset A\subset M$ be finite von Neumann algebras with a fixed trace. 
We identify $\langle M,A\rangle= \overline{Me_AM}^{\rm w} \subset \langle M,B\rangle$ and $\langle A, B\rangle=\overline{Ae_BA}^{\rm w} \subset \langle M,B\rangle$ and notice that $\mathrm{Tr}_{\langle M,B\rangle}|_{\langle A, B\rangle}=\mathrm{Tr}_{\langle A,B\rangle}$. 
Assume that $B\subset A$ is a finite extension. Then since $\mathrm{Tr}_{\langle M,B\rangle}$ is still semifinite on $\langle M,A\rangle$ (because $\mathrm{Tr}_{\langle M,B\rangle}(e_A)=\mathrm{Tr}_{\langle A,B\rangle}(e_A)<\infty$), there is a conditional expectation from $\langle M,B\rangle$ onto $\langle M,A\rangle$. 

\begin{Rem}\upshape\label{intertwiner cross1.5}
Let $\pi$ and $e$ be as in the previous proposition. 
When $A$ is a $\rm II_1$ factor, $\pi$ is defined on $N$ as a unital and normal one. The inclusion $\pi(A)\subset A\mathbin{\bar{\otimes}}\mathbb{B}(\ell^2(\mathcal{S}))$ is a finite extension. 
When $\mathcal{Z}(A)$ is diffuse, if we further assume $N$ and $A\rtimes\Gamma_0$ are $\rm II_1$ factors and $pAp\subset N$ is regular, then we can extend $\pi$ on $N$ in the following way. 
We first extend $\pi$ on $pAp$ by $\pi(a):=\sum_{s\in\mathcal{S}} \alpha_{\tilde{s}}^{-1}(a)\otimes e_{s,s}$. 
Exchanging $e$ with a small one, we may assume $e$ has trace $1/m$ in $N$. 
Then $\pi(e)=\sum_{s\in\mathcal{S}}\alpha_{\tilde{s}}^{-1}(e)\otimes e_{s,s}\in \pi(p)(\mathcal{Z}(A)\mathbin{\bar{\otimes}} \mathbb{B}(\ell^2(\mathcal{S})))\pi(p)$ has trace $1/m$ with $\mathrm{Tr}_A\otimes \mathrm{Tr}_{\ell^2(\mathcal{S})}$, where $\mathrm{Tr}_{\ell^2(\mathcal{S})}$ is the normalized trace. 
Let $e_1:=e$ and $e_i\in \mathcal{Z}(A)p$ $(i=2,\ldots,m)$ be mutually orthogonal projections having trace $1/m$. We have $p=\sum_{i=1}^{m}e_i$. 
By Lemma \ref{projections in Cartan}, take partial isometries $v_i\in N$ and $w_i\in \pi(p)(A\rtimes \Gamma_0\mathbin{\bar{\otimes}} \mathbb{B}(\ell^2(\mathcal{S})))\pi(p)$ $(i=1,\ldots,m)$ satisfying that 
$v_iv_i^*=e_1$, $w_iw_i^*=\pi(e_1)$, $v_i^*v_i=e_i$, $w_i^*w_i=\pi(e_i)$, $v_iAv_i^*= Ae_1$ and $w_i(A\mathbin{\bar{\otimes}} \ell^\infty(\mathcal{S}))w_i^*=(A\mathbin{\bar{\otimes}} \ell^\infty(\mathcal{S}))\pi(e_1)$. 
Define $\tilde{\pi}(x):=\sum_{i,j=1}^{m} w_i^* \pi(v_i xv_j^*) w_j$ for $x\in N$. Then $\tilde{\pi}$ is a unital normal $*$-homomorphism from $N$ into $\pi(p)(A\rtimes \Gamma_0\mathbin{\bar{\otimes}} \mathbb{B}(\ell^2(\mathcal{S})))\pi(p)$ satisfying $\tilde{\pi}=\pi$ on $eNe$. 
Since $\pi(eAe)\subset \pi(e) (A\mathbin{\bar{\otimes}}\mathbb{B}(\ell^2(\mathcal{S}))) \pi(e)$ is a finite extension, $\pi(pAp)\subset \pi(p)(A\mathbin{\bar{\otimes}}\mathbb{B}(\ell^2(\mathcal{S})))\pi(p)$ is also a finite extension. 
\end{Rem}

\begin{Cor}\label{intertwiner cross2}
Assume either that $\alpha|_{\Gamma_0}$ is ergodic on $\mathcal{Z}(A)$ so that $A\rtimes\Gamma_0$ is a $\rm II_1$ factor, or $\alpha$ is free on $\mathcal{Z}(A)$. 
If $N\preceq_M A\rtimes \Gamma_0$, 
then there exist non-zero projections $e,f\in\mathcal{Z}(A)$ (or $e,f\in A$ when $A$ is a $\rm II_1$ factor) with $e\leq p$, a normal unital $*$-homomorphism $\theta\colon eNe \rightarrow f(A\rtimes \Gamma_0)f$, and a non-zero partial isometry $v\in e(A\rtimes\Gamma)f$ such that 
$xv=v\theta(x)$ for $x\in eNe$, $vv^*\in \mathcal{Z}(A)e\cap N'$, and $\theta(eAe)\subset fAf$. 
In the case, the inclusion $\theta(eAe)\subset fAf$ is a finite extension.
\end{Cor}
\begin{proof}
Let $e$, $d$, $\mathcal{S}$, and $\pi$ be as in the previous proposition. 
Write $\mathcal{S}=\{s_1,\ldots,s_n\}$ and write the matrix unit of $\mathbb{B}(\ell^2(\mathcal{S}))$ along $s_i$ as $(e_{i,j})_{i,j}$. 
Assume first $\mathcal{Z}(A)$ is diffuse and $A\rtimes\Gamma_0$ is a $\rm II_1$ factor. 
Then there are projections $e_i\in\mathcal{Z}(A)$ such that $e=\sum_{i=1}^ne_i$ and each $e_i$ has the same trace in $A\rtimes\Gamma_0$. 
We may assume $de_1\neq 0$ so that $\pi(e_1)\neq 0$. 
By the proof of Corollary \ref{Popa embed4} (regarding $e_1Ne_1\simeq e_1Ne_1\mathbin{\bar{\otimes}}\mathbb{C} e_{1,1}$), there is a partial isometry $w\in (e_1\otimes e_{1,1})(M\mathbin{\bar{\otimes}} \mathbb{M}_n)$ such that $(x\otimes e_{1,1})w=w\pi(x)$ for any $x\in e_1Ne_1$. 
This equation implies $aw_j=w_j\alpha_{\tilde{s}_j}^{-1}(a)$ for all $a\in A$, where we write $w=\sum_{j}w_{j}\otimes e_{1,j}$. 
By a similar manner to that in the proof of the last lemma, we get $w_j=a_j\lambda_{\tilde{s}_j}$ for some $a_j\in e_1\mathcal{Z}(A)$. 
By Lemma \ref{projections in Cartan}, take partial isometries $u_i$ in $A\rtimes\Gamma_0$ such that $u_i^*u_i= e_i$, $u_iu_i^*=\alpha_{\tilde{s}_i}^{-1}(e_1)$, and $u_i^*Au_i=e_iA$. 
Put $U:=(\delta_{i,j}u_i)_{i,j}$. 
Let $V$ be a partial isometry in $(A\rtimes\Gamma_0) \mathbin{\bar{\otimes}} \mathbb{B}(\ell^2(\mathcal{S}))$ given by $V_{i,j}=e_i\delta_{1,j}$ and note  $VV^*= U^*\pi(e_1)U$. 
Put $\theta:=\mathrm{Ad}(V^*U^*)\circ\pi\colon e_1Ne_1\rightarrow A\rtimes \Gamma_0\mathbin{\bar{\otimes}} \mathbb{C}e_{1,1}\simeq A\rtimes \Gamma_0$, $\theta(e_1)=\sum_{i=1}^ne_i=:f$ and  $wUV\simeq \sum_i a_i\lambda_{\tilde{s}_i}u_i=:v$. 
We have $v\theta(x)=xv$ for $x\in e_1Ne_1$ and hence $vv^*\in (e_1Ne_1)'$. 
Also we have $vv^*=\sum_ia_i\lambda_{\tilde{s}_i}u_iu_i^*\lambda_{\tilde{s}_i}^*a_i^*
=\sum_ia_i\lambda_{\tilde{s}_i}\alpha_{\tilde{s}_i}^{-1}(e_1)\lambda_{\tilde{s}_i}^*a_i^*=\sum_ia_ia_i^*\in e_1\mathcal{Z}(A)$ and 
$\theta(a)=\sum_{i=1}^n   u_i^*\alpha_{\tilde{s}}^{-1}(a) u_i  \in fAf$ for $a\in Ae_1$. 
Finally $\theta(e_1Ae_1)\subset fAf$ is a finite extension, since so is $\mathrm{Ad}U^*\circ\pi(e_1Ae_1)\subset E(A\mathbin{\bar{\otimes}}\mathbb{B}(\ell^2(\mathcal{S})))E$, where $E:=\mathrm{Ad}U^*\circ\pi(e_1)$. 

Next assume that $A$ is a $\rm II_1$ factor so that $\mathrm{Tr}_A$ is finite and $p=1$. Then decomposing $e=1_A=\sum_{i=1}^ne_i$ for $e_i\in A$ with $e_i\sim e_j$ in $A$, we can take $u_i,U,V,v,\pi$, and $\theta$ as above, which do the work.

Finally assume $\mathcal{Z}(A)$ is diffuse and $\alpha$ is free on $\mathcal{Z}(A)$. 
By Lemma \ref{free}, there is $e_1\leq e$ such that $\alpha_{\tilde{s}}^{-1}(e_1)\alpha_{\tilde{t}}^{-1}(e_1)=0$ for any $s,t\in\mathcal{S}$ with $s\neq t$. 
In the case, we do not need to take $u_i$ above. So putting $\theta:=\mathrm{Ad}V^*\circ\pi|_{e_1Ne_1}$, we are done. 
\end{proof}

%%%%%%%%%%%%%%%%%%%%%%%%%%%%%%%%%%%%
\subsection{\bf Proof of Theorem \ref{B}}
%%%%%%%%%%%%%%%%%%%%%%%%%%%%%%%%%%%%

In the subsection, we use the same notation as in Theorem \ref{B}. We write $M:=A\rtimes \Gamma$, $\Gamma_X:=\prod_{i\in X}\Gamma_i$ and $\Lambda_Y:=\prod_{j\in Y}\Lambda_j$ for $Y\subset \{1,\ldots,n\}$ and $X\subset \{1,\ldots,m\}$. 

\begin{Lem}\label{reduction lemma2}
For any subset $Y\subset \{1,\ldots,n\}$ with $|Y|\leq m$, there is $X\subset \{1,\ldots,n\}$ such that $|X|=|Y|$ and $B\rtimes\Lambda_Y\preceq_{M} A\rtimes\Gamma_X$.
\end{Lem}
\begin{proof}
We apply Proposition \ref{bi-exact theorem} to $B\rtimes (\Lambda_2\times\cdots\times\Lambda_n)\subset pMp$. 
If there is an expectation from $p\langle M,A\rangle p$ onto $L(\Lambda_2\times\cdots\times\Lambda_n)'$ which contains $L\Lambda_1$, then we have a contradiction when $A$ is amenable because $\langle M,A\rangle$ is amenable. 
When $A$ is non-amenable and $p=1$, we have an expectation from $\mathbb{C}1_B\mathbin{\bar{\otimes}} \mathbb{B}(\ell^2(\Lambda))\subset \langle B\rtimes\Lambda,B\rangle\subset \langle M,A\rangle$ into $L\Lambda_1$ and hence a contradiction. 
So we have $B\rtimes(\Lambda_2\times\cdots\times\Lambda_n)\preceq_{M}A\rtimes \Gamma_{X_1}$ for some $X_1:=\{1,\ldots, m\}\setminus \{i\}$. 
For simplicity we assume $i=1$. 

We take $e,d,\mathcal{S}$, and $\pi$ in Proposition \ref{intertwiner cross} and write $\mathbb{B}(\ell^2(\mathcal{S}))=:\mathbb{M}_{m_1}$ and $\tilde{A}:=A\mathbin{\bar{\otimes}}\mathbb{M}_{m_1}$. 
Considering the trivial action of $\Gamma_{X_1}$ on $\mathbb{M}_{m_1}$, we regard $(A\rtimes \Gamma_{X_1})\mathbin{\bar{\otimes}} \mathbb{M}_{m_1}=\tilde{A}\rtimes\Gamma_{X_1}$. 
Here we claim that $\pi(e(B\rtimes (\Lambda_3\times\cdots\times\Lambda_n))e)\preceq_{\pi(e)(\tilde{A}\rtimes \Gamma_{X_1})\pi(e)} \pi(e)(\tilde{A}\rtimes \Gamma_{X_2})\pi(e)$ for some $X_2:=X_1\setminus \{i\}$ (we assume $i=2$ for simplicity). 
By Remark \ref{intertwiner cross1.5}, we extend $\pi$ on $B\rtimes(\Lambda_2\times\cdots\times\Lambda_n)$ and apply again Proposition \ref{bi-exact theorem} to $\pi(B\rtimes (\Lambda_3\times\cdots\times\Lambda_n))\subset \tilde{A}\rtimes \Gamma_{X_1}$. 
If there is an expectation from $\pi(p)\langle \tilde{A}\rtimes \Gamma_{X_1}, \tilde{A}\rangle \pi(p)$ onto $\pi(L(\Lambda_3\times\cdots\times\Lambda_n))'$ which contains $\pi(L\Lambda_2)$, then we have a contradiction when $A$ is amenable. 
When $A$ is non-amenable and $p=1$, 
since $\pi(B)\subset \tilde{A}$ is a finite extension, there is an expectation from $\langle \tilde{A}\rtimes \Gamma_{X_1}, \pi(B)\rangle$ onto $\langle \tilde{A}\rtimes \Gamma_{X_1}, \tilde{A}\rangle$. 
So we have an expectation from $\langle \pi(B\rtimes(\Lambda_2\times\cdots\times\Lambda_n)),\pi(B)\rangle$, which is a subalgebras of $\langle \tilde{A}\rtimes \Gamma_{X_1}, \pi(B)\rangle$, onto $\pi(L\Lambda_2)$. This contradicts to the amenability of $\Lambda_2$. 
Thus we get $\pi(B\rtimes (\Lambda_3\times\cdots\times\Lambda_n))\preceq_{\tilde{A}\rtimes \Gamma_{X_1}} \tilde{A}\rtimes \Gamma_{X_2}$ for some $X_2$. 
By Corollary \ref{Popa embed3}, we get the claim. 

Now by construction, we are in fact seeing $e(B\rtimes(\Lambda_2\times\cdots\times\Lambda_n))ed\subset \pi(e)(A\rtimes \Gamma_{X_1}\mathbin{\bar{\otimes}}\mathbb{B}(\ell^2(\mathcal{S})))\pi(e)$. 
So by the same manner as in the proof of Lemma \ref{reduction lemma}, we can deduce $B\rtimes(\Lambda_3\times\cdots\times\Lambda_n)\preceq_{M}A\rtimes \Gamma_{X_2}$. 
This completes the proof. 
\end{proof}

\begin{proof}[\bf Proof of Theorem \ref{B} (first half)]
Suppose by contradiction that $n>m$. 
Then by the previous lemma, we have $B\rtimes\Lambda_Y\preceq_M A$ for some $Y\neq \emptyset$. 
By Proposition \ref{intertwiner cross} and Remark \ref{intertwiner cross1.5}, we have  a $*$-homomorphism $\pi\colon B\rtimes\Lambda_Y\rightarrow A\mathbin{\bar{\otimes}}\mathbb{M}_n$ for some $n$ such that $\pi(B)\subset \pi(p)(A\mathbin{\bar{\otimes}} \mathbb{M}_n)\pi(p)$ is a finite extension. 
Then since $\pi(B)$ is co-amenable in $\pi(p)(A\mathbin{\bar{\otimes}} \mathbb{M}_n)\pi(p)$, $\pi(B)$ is co-amenable in $\pi(B\rtimes\Lambda_Y)$. 
This contradicts to the non-amenability of $\Lambda_Y$. 
\end{proof}

By Lemma \ref{reduction lemma2}, if $n=m$, for any $i$ there is some $j$ such that $B\rtimes \Lambda_i\preceq_MA\rtimes \Gamma_j$. We show that the assignment $i\mapsto j$ is one to one.

\begin{Lem}
If $B\rtimes\Lambda_i\preceq_{M}A\rtimes \Gamma_j$ and $B\rtimes\Lambda_i\preceq_{M}A\rtimes \Gamma_l$ for some $i$, $j$ and $l$, then we have $j=l$.
\end{Lem}
\begin{proof}
Suppose $j\neq l$. 
Since $(B\rtimes\Lambda_i)' \cap pMp=\mathbb{C}p$, 
by Lemma \ref{2.5}, we have $B\rtimes\Lambda_i\subset_{\rm approx}A\rtimes \Gamma_a$ for $a=j,l$. Hence we have $B\rtimes\Lambda_i\subset_{\rm approx} A$ by Lemma \ref{2.7}. This implies $B\rtimes\Lambda_i\preceq_M A$ and it contradicts to (the proof of) the first part of Theorem \ref{B}.  
\end{proof}

The proof of the following lemma was from that of \cite[\textrm{Lemma 33}]{Sa09}. 

\begin{Lem}
If $B\rtimes\Lambda_i\preceq_{M}A\rtimes \Gamma_j$ and $B\rtimes\Lambda_k\preceq_{M}A\rtimes \Gamma_j$ for some $i$, $j$ and $k$, then we have $i=k$.
\end{Lem}
\begin{proof}
By assumption, there are non-zero trace finite projections $e_a\in (B\rtimes\Lambda_a)' \cap ((A\rtimes\Gamma_j)^{\rm op})'$ with $e_a= e_ap$ for $a=i,k$. 
Observe that for any $s\in \Lambda_{\{i\}^c}$ and $g\in \Gamma_{\{j\}^c}$, the element $\rho_g\lambda_s e_i \lambda_s^*\rho_g^*$ satisfies the same condition as that on $e_i$. 
Let $e$ be the element $\sup_{s\in \Lambda_{\{i\}^c},g\in \Gamma_{\{j\}^c}}\rho_g\lambda_s e_i \lambda_s^*\rho_g^*$. Then $e$ is contained in 
\begin{eqnarray*}
&& (L\Lambda_{\{i\}^c})'\cap(B\rtimes\Lambda_i)'p \cap  ((L\Gamma_{\{j\}^c})^{\rm op})'\cap((A\rtimes\Gamma_j)^{\rm op})'\\
&=& (B\rtimes\Lambda)'p \cap ((A\rtimes\Gamma)^{\rm op})'\\
&=& (B\rtimes\Lambda)'p \cap (A\rtimes\Gamma)\\
&=& (B\rtimes\Lambda)' \cap p\mathcal{Z}(A)= L\Lambda'\cap \mathcal{Z}(B)=\mathbb{C}p.
\end{eqnarray*}
Hence we have $e=p$. Thus there exist finite subsets $\mathcal{E} \subset \Lambda_{\{i\}^c}$ and $\mathcal{F}\subset \Gamma_{\{j\}^c}$ satisfying that $\vee_{s\in \mathcal{E},g\in\mathcal{F}}\rho_g\lambda_s e_i \lambda_s^*\rho_g^*$ is not orthogonal to $e_k$. 
By exchanging $e_i$ with this element, we can assume $e_ie_k\neq0$. 

Suppose now $i\neq k$. 
We claim that $B\rtimes(\Lambda_i\times\Lambda_k)\preceq_{M}A\rtimes \Gamma_j$. 
Consider the $\sigma$-weak closure of $\mathrm{co}\{\lambda_se_i\lambda_s^*\mid s\in \Lambda_k \}$ and, 
regarding this set as a subset of $L^2(\langle A\rtimes \Gamma, A\rtimes \Gamma_j\rangle)$, take the circumcenter $d$, which is contained in $L\Lambda_k'\cap (B\rtimes\Lambda_i)'p \cap ((A\rtimes\Gamma_j)^{\rm op})'=(B\rtimes (\Lambda_i\times\Lambda_k))'p\cap ((A\rtimes\Gamma_j)^{\rm op})'$. 
This is non-zero since we have for any $s\in\Lambda_k$,
\begin{eqnarray*}
 \langle\lambda_se_i\lambda_s^* , e_k\rangle
=\mathrm{Tr}_{\langle A\rtimes \Gamma, A\rtimes \Gamma_i\rangle}(\lambda_se_i\lambda_s^*e_k)
=\mathrm{Tr}_{\langle A\rtimes \Gamma, A\rtimes \Gamma_i\rangle}(e_ie_k)>0.
\end{eqnarray*}
So we get the claim. 

Now by the proof of Lemma \ref{reduction lemma2},  we have $B\rtimes\Lambda_i\preceq_{M}A$. This contradicts to (the proof of) the first part of Theorem \ref{B}. 
\end{proof}

Thanks for previous two lemmas, the assignment $i\mapsto j$ above gives a bijective map on $\{1,\ldots,n\}$. Putting $j=\sigma(i)$, we complete the proof.

%%%%%%%%%%%%%%%%%%%%%%%%%%%%%%%%%%%%
\section{\bf Another approach to prime factorization results}
%%%%%%%%%%%%%%%%%%%%%%%%%%%%%%%%%%%%

%%%%%%%%%%%%%%%%%%%%%%%%%%%%%%%%%%%%
\subsection{\bf Irreducibility and primeness}
%%%%%%%%%%%%%%%%%%%%%%%%%%%%%%%%%%%%

In the number theory, there are two notions of prime numbers. Recall that a number $p\in\mathbb{N}$ is irreducible if for any $q,r\in\mathbb{N}$ with $p=qr$, we have $q=1$ or $r=1$; and is prime if for any $q,r,s\in\mathbb{N}$ with $pq=rs$, we have $p\mid r$ or $p\mid s$. 
In our von Neumann algebra theory, we used irreducibility as a definition of primeness for von Neumann algebras. If we adopt primeness of the number theory, the following condition should be a corresponding notion:
\begin{itemize}
	\item We say a $\rm II_1$ factor $M$ is ``prime'' if for any $\rm II_1$ factor $N,K,L$ with $M\mathbin{\bar{\otimes}} N=K\mathbin{\bar{\otimes}} L$, there is a unitary $u\in\mathcal{U}(M)$ and $t>0$ such that 
$uM u^*\subset K^t$ or $uMu^*\subset L^t$.
\end{itemize}
We prove that there are such examples. 

\begin{Thm}
Let $\Gamma$ be a discrete group. Assume that $\Gamma$ is non-amenable, ICC, bi-exact and weakly amenable. 
Then for any $\rm II_1$ factor $B$, $K$ and $L$ with $L\Gamma\mathbin{\bar{\otimes}} B=K\mathbin{\bar{\otimes}} L(=:M)$, we have either $L\Gamma\preceq_M K$ or $L\Gamma\preceq_M L$. 
If $L\Gamma\preceq_M K$, then there is a unitary $u\in\mathcal{U}(M)$ and $t>0$ such that 
$uL\Gamma u^*\subset K^t$.
\end{Thm}
\begin{proof}
Suppose by contradiction that $L\Gamma\not\preceq_M K$ and $L\Gamma\not\preceq_M L$. 
By \cite[\textrm{Lemma 3.5}]{Va08}, this exactly means $K \not\preceq_M B$ and $L \not\preceq_M B$. By \cite[\textrm{Corollary F.14}]{BO08}, there is a diffuse abelian subalgebra $A\subset K$ such that $A \not\preceq_M B$. 
By \cite[\textrm{Theorem 1.4}]{PV12}, $\mathcal{N}_M(A)''$ is amenable relative to $B$ in M. Since $L\subset\mathcal{N}_M(A)''$, $L$ is also amenable relative to $B$. We again apply \cite[\textrm{Theorem 1.4}]{PV12} to $L$ and get that $M=\mathcal{N}_M(L)''$ is amenable relative to $B$. 
This means $L\Gamma$ is amenable and hence a contradiction. 
The last assertion follows from Lemma \ref{hariawase1}.
\end{proof}

\begin{Rem}\upshape
Since we generalized \cite[\textrm{Theorem 1.4}]{PV12} to quantum groups of Kac type \cite[\textrm{Theorem A}]{Is13}, 
the same thing is true for $\rm II_1$ factors of $L^\infty(\mathbb{G}_i)$, where $\hat{\mathbb{G}}_i$ is non-amenable, bi-exact and weakly amenable. 
\end{Rem}

Once we get the property, it is easy to deduce the following prime factorization results. 
Since proofs are straightforward, we leave it to the reader.

\begin{Cor}
Let $M_i$ $(i=1,\ldots,m)$ be $\rm II_1$ factors. Assume that each $M_i$ is ``prime'' in the above sense. Let $M_0$ and $N_j$ $(j=0,1,\ldots,n)$ be prime $\rm II_1$ factors in the usual sense satisfying 
$M_0\mathbin{\bar{\otimes}} M_1\mathbin{\bar{\otimes}} \cdots \mathbin{\bar{\otimes}} M_m= N_0\mathbin{\bar{\otimes}} \cdots \mathbin{\bar{\otimes}} N_n (=:M)$. 
Then $n=m$ and there are a unitary $u\in \mathcal{U}(M)$, $\sigma\in  \mathfrak{S}_{n+1}$, and $t_i>0$ with $t_0\cdots t_n=1$ such that $uM_iu^*=N_{\sigma(i)}^{t_i}$.
\end{Cor}

\begin{Cor}
Let $M_i$ $(i=1,\ldots,m)$, $M_0$, and $M$ be as in the previous corollary. Then $\mathcal{F}(M)=\mathcal{F}(M_0)\mathcal{F}(M_1)\cdots\mathcal{F}(M_m)$. Here $\mathcal{F}(M)$ and $\mathcal{F}(M_i)$ are fundamental groups of $M$ and $M_i$.
\end{Cor}

\end{document}